\documentclass[11pt,reqno]{amsart}
\usepackage{amssymb,amscd,verbatim, amsthm,graphicx, color}
\usepackage[mathscr]{eucal}
\usepackage[all]{xy}

\newcommand \fk[1]{{{\mathfrak #1}}}
\newcommand \C[1]{{\mathcal #1}}
\newcommand \ovl[1]{{\overline {#1}}}

\newcommand \wti[1]{{\widetilde {#1}}}
\newcommand \wht[1]{{\widehat {#1}}}

\newcommand\fg{\mathfrak g}

\newcommand \bA{{\mathbb A}}
\newcommand \bC{{\mathbb C}}
\newcommand \bF{{\mathbb F}}
\newcommand \bH{{\mathbb H}}
\newcommand \bR{{\mathbb R}}
\newcommand \bZ{{\mathbb Z}}

\newcommand \bO{{\mathbb O}}

\newcommand\cha{{\check \alpha}}

\newcommand\one{\mathrm 1}

\newcommand\CA{{\C A}}

\newcommand\CG{{\C G}}
\newcommand\CH{{\C H}}
\newcommand\CI{{\C I}}
\newcommand\CO{{\C O}}

\newcommand\CP{{\C P}}
\newcommand\CK{{\C K}}

\newcommand\bW{{\mathbf W}}

\newcommand\ie{{\it i.e.~}}
\newcommand\cf{{\it cf.~ }}

\newcommand\la{{\lambda}}
\newcommand\om{{\omega}}
\newcommand\al{{\alpha}}
\newcommand\sig{{\sigma}}

\newcommand\fh{{\mathfrak h}}
\newcommand\ft{{\mathfrak t}}

\newtheorem{proposition}{Proposition}[subsection]
\newtheorem{corollary}[proposition]{Corollary}
\newtheorem{lemma}[proposition]{Lemma}
\newtheorem{theorem}[proposition]{Theorem}

\theoremstyle{definition}
\newtheorem{definition}[proposition]{Definition}
\newtheorem{remark}[proposition]{Remark}
\newtheorem{remarks}[proposition]{Remarks}
\newtheorem{example}[proposition]{Example}
\newtheorem{notation}[proposition]{Notation}

\newcommand\Hom{\operatorname{Hom}}
\newcommand\Ind{\operatorname{Ind}}
\newcommand\Aut{\operatorname{Aut}}
\newcommand\cInd{\operatorname{c-Ind}}
\newcommand\End{\operatorname{End}}
\newcommand\Gal{\operatorname{Gal}}

\newcommand\pr{\operatorname{pr}}
\newcommand\tr{\operatorname{tr}}

\newcommand\triv{\mathrm{triv}}

\numberwithin{equation}{subsection}

\begin{document}

\title{Unitary equivalences for reductive $p-$adic groups} 
\author{Dan Barbasch}
       \address[D. Barbasch]{Dept. of Mathematics\\
               Cornell University\\Ithaca, NY 14850}
       \email{barbasch@math.cornell.edu}

\author{Dan Ciubotaru}
        \address[D. Ciubotaru]{Dept. of Mathematics\\ University of
          Utah\\ Salt Lake City, UT 84112}
        \email{ciubo@math.utah.edu}

\date{\today}

\begin{abstract}
We establish a transfer of unitarity for a Bernstein
component of the category of smooth representations of a reductive
$p$-adic group to the associated Hecke algebra, in the framework of
the theory of types, whenever the Hecke algebra is an affine Hecke algebra with
geometric parameters, in the sense of Lusztig (possibly extended
by a group of automorphisms of the root datum). It is known that there
is a large class of such examples (detailed in the paper). As a
consequence, we establish relations between the unitary duals of
different groups, in the spirit of endoscopy.
\end{abstract}

\maketitle

\setcounter{tocdepth}{1}

\begin{small}
\tableofcontents
\end{small}

\section{Introduction}\label{sec:1}

\subsection{}{
A central result in the representation
theory of $p-$adic groups is the Borel-Casselman equivalence \cite{Bo}
between the category of smooth 
representations with Iwahori fixed vectors, and the category of 
modules of the Iwahori-Hecke algebra. This result  has
served as a paradigm for subsequent efforts to solve the problem of classifying
the smooth dual of a $p-$adic group, \cite{HM}, \cite{BK1}, \cite{K},
\cite{Ro}, \cite{Se} and many others. From the outset it was
  conjectured that the Borel-Casselman correspondence also preserves
  unitarity.  In \cite{BM1}, it is proved that, for a
split group $G$ of adjoint type over a $p$-adic field $\bF$, the
Borel-Casselman equivalence between the categories of
representations with Iwahori fixed vectors, and finite dimensional
modules of the affine Hecke algebra preserves unitarity. The results depend
heavily on the fact that the affine Hecke algebras has equal
parameters, and corresponds to an adjoint group. In addition, one has
to make a certain technical assumption, namely that the infinitesimal
character  be real (\cf \ref{sec:consequence}). This assumption is
removed in \cite{BM2}. 
 
The first purpose of this paper is to establish 
the preservation of unitarity in a much more general setting. 
To do this, we treat the case of an extended affine Hecke
algebra with general parameters of geometric type, and coming from a
group of arbitrary isogeny class. The work of \cite{BM1} relies on combining the
classification of representations of the affine Hecke algebra and the
idea of signature character in \cite{V1}. In order to treat the
larger class of Hecke algebras in this paper, we have to 
analyze the relation between the
affine Hecke algebra and the graded affine Hecke algebra as in
\cite{BM2}, and also rely heavily on the
classification of irreducible modules of the 
graded affine Hecke algebra in (\cite{L1,L2,L3,L4}).  

The main technical result that we prove in this direction is summarized in  the
following theorem. Let $\Psi=(X,X^\vee,R,R^\vee)$ be a
  root datum corresponding to a complex connected reductive group
  $G(\Psi).$ Let 
  $\C H^{\lambda,\lambda^*}(\Psi,z_0)$ be the affine Hecke algebra attached to 
$\Psi$ (Definition \ref{d:2.1}) and with parameters
$\lambda,\lambda^*$  of geometric   type (section \ref{sec:3}). We
assume that $z_0$ is not a root of   unity. Let $\C H'(\Psi,z_0)=\C 
H^{\lambda,\lambda^*}(\Psi,z_0){\rtimes}\Gamma$ be an
extended affine   Hecke algebra by a finite group $\Gamma$ acting by 
automorphisms of $\Psi$ (Definition \ref{d:2.1}). 

Let $s_e\in T_e:=X^\vee\otimes_\bZ S^1$ be a
  fixed elliptic semisimple element in $G(\Psi)$.}
\begin{theorem}\label{t:vogan}   The
  category of finite dimensional $\C
  H'(\Psi,z_0)-$modules whose central characters have
  elliptic parts $G(\Psi){\rtimes} \Gamma$-conjugate to $s_e$ has the
  Vogan property   (Definition \ref{d:vogan}).
\end{theorem}
\subsection{}Theorem \ref{t:vogan} plays the same role as its counterparts in  \cite{BM1}
and \cite{BM2}; namely it provides a means to conclude that whenever a
finite dimensional representation of the Hecke algebra is unitary,
then the corresponding infinite dimensional admissible representation in the
Bernstein component is also unitary. This is the content of Theorem \ref{t:types} below. The details are in section \ref{sec:5}.

 Let $(J,\rho)$ be a type for the
element $\mathfrak s=[L,\sigma]$ in the Bernstein center ($L$ a Levi subgroup,
$\sigma$ a supercuspidal representation of $L$). Let $\mathfrak s_L$
be the Bernstein class for $L$ defined by $(L,\sigma)$ and assume that
$\mathfrak s_L$ has a type $(J_L,\rho_L)$ such that $(J,\rho)$ is a
cover of $(J_L,\rho_L)$ in the sense of \cite[Definition
8.1]{BK2}. Let $\C H(\CG,\rho)$ and $\C H(L,\rho_L)$ be the Hecke
algebras (Definition \ref{d:startype}). Let $P$ be a parabolic
subgroup with Levi subgroup $L$. Let $t_P:\C
H(L,\rho_L)\to \C H(\CG,\rho)$ be the embedding defined in  
\cite[Corollary 7.12]{BK2}. 

\begin{theorem}\label{t:types}Let $\mathfrak s=[L,\sigma]$ be an
  element in the Bernstein center of $\C G$ and
  $\fk R^{\fk s}(\C G)$ the corresponding Bernstein component. Assume that there exists a type
  $(J,\rho)$ (Definition \ref{d:type}) for $\fk R^{\fk s}(\C G)$ which
  satisfies the following conditions\footnote{In practice, when one calculates the isomorphism in (I), the compact
subgroup $\C K$ in (II) is implicitly determined, and condition (III)
is verified.}:

\begin{enumerate} 
\item[(I)] the algebra $\C H(\CG,\rho)$ is isomorphic to
an extended affine Hecke algebra $\C H'(\Psi,z_0)$ with geometric parameters, such that the isomorphism preserves the natural
Hilbert algebra structures (as defined in section \ref{sec:2.6} and Definition \ref{d:startype});
\item[(II)] there exists a compact open subgroup $\C K\supset J$ of
  $\C G$ with $\C G=\C K P$, such that the isomorphism in (I) maps the  Hecke algebra $\C
  H(\C K,\rho)$ to the finite Hecke subalgebra $\C H_{W'}$ of $\C
  H'(\Psi,z_0)$ (as defined in section \ref{sec:2.2}).  
\item[(III)] the isomorphism in (I) maps $t_P(\C H(L,\rho_L))$ onto
  $\C A(\Psi)$ (as defined in section \ref{sec:2.2}).
\end{enumerate}

Then the equivalence of categories $\mathbf M_\rho: \fk R^{\fk s}(\C G)\to \C H(\CG,\rho)\text{-mod}$ given by the theory of types (see (\ref{d:Mrho})) induces a bijection between the irreducible unitary modules in $\fk R^{\fk s}(\C G)$ and $\C H(\CG,\rho)\text{-mod}.$
\end{theorem}

Conditions (I)-(III), and thus Theorem \ref{t:types}, are expected to hold in great generality.  They are known to hold at least in the following cases: 
\begin{enumerate}
\item unramified principal series (\cite{Bo},\cite{IM},\cite{Ti});
\item {$GL(n,\bF)$ (\cite{BK1}) and $GL(n,D)$ where $D$ is a division
    algebra over $\bF$ (\cite{Se});} 
\item unipotent representations of simple groups of adjoint type
  (\cite{L6,L8,L4}); 
\item ramified principal series of split groups (\cite{Ro});
\item pseudo-spherical principal series for double covers of split
   groups of simply-connected type (\cite{LS}).
\end{enumerate}

{An important  consequence of this approach is that it
  gives a unified 
framework for obtaining unitary 
categorical equivalences between two Bernstein components for possibly
different groups. Whenever two Bernstein components are controlled by
isomorphic affine Hecke algebras such that the conditions (I)-(III) above
hold, one has a  unitary equivalence. The example we
present in detail to illustrate this is the case of ramified principal
series of a split group where we obtain a correspondence with the
unramified principal series of a (split) endoscopic group (Theorem \ref{t:1.4}). 

There
is another, subtler phenomenon: there are important cases when the two
affine Hecke algebras are not isomorphic, but certain of their graded versions
(in the sense of \cite{L1}) are. Again our methods allow us to
conclude similar unitary correspondences. The example that we present
in detail is that of unramified principal series of quasisplit
unramified groups, and the equivalences are again with unramified
principal series of certain (split) endoscopic groups (Theorem \ref{t:1.3}). We explain these
two examples in 
more detail in the introduction, after presenting the fundamental case
of the Borel-Casselman equivalence.} 

\subsection{}
Let $\bF$ be a $p$-adic field of characteristic zero. Let $\CG$ be the
$\bF$-points of a connected linear algebraic reductive group defined
over $\bF.$ Let $A$ be a maximally split torus of $\CG$ and set
$M=C_\CG(A)$, the centralizer of $A$ in $\CG.$ Fix a special
maximal compact open subgroup $\CK$ and an Iwahori subgroup
$\CI\subset \CK$ of $\CG$, attached by Bruhat-Tits
theory to $A.$ Let $^0\!M=M\cap \CK$ be the maximal compact open
subgroup of $M.$ A complex character $\chi: M\to \bC^\times$ is called
unramified if $\chi|_{^0\!M}=\one,$ and it is called ramified otherwise.

Fix $^0\!\chi$ a character of $^0\!M$, and consider the category $\fk
R^{^0\!\chi}(\CG)$ of smooth $\CG$-representations, which
appear as constituents of the minimal principal series induced from
complex characters $\chi$ of $M$ such that $\chi|_{^0\!M}=~^0\!\chi.$ This
category is a Bernstein component of the category of all smooth
 $\CG$-representations. We are interested in the study of
hermitian and unitary modules in $\fk
R^{^0\!\chi}(\CG).$ 

The first case is when $^0\!\chi=\one,$ the unramified
principal series. In this case, the category $\fk
R^{\one}(\CG)$ is known to be naturally equivalent with the category of modules over the Iwahori-Hecke algebra $\CH(\CG,\one_\CI)$
 (\cite{Bo}). The algebra $\CH(\CG,\one_\CI)$
is an affine Hecke algebra, possibly with unequal parameters
(Definition (\ref{d:2.1})), and has a natural $*$-operation (section
\ref{sec:2.6}); therefore one can define hermitian and unitary
modules.  

\begin{theorem}In the Borel-Casselman equivalence $\fk
R^{\one}(\CG)\cong \CH(\CG,\one_\CI)$-mod, the hermitian and unitary
representations correspond, respectively. 
\end{theorem}

In this theorem, the group $\CG$ is of arbitrary isogeny and not
necessarily split, and we
emphasize that the correspondence is via the equivalence of categories. The case when $\CG$
is adjoint and split over $\bF$ is in \cite{BM1,BM2}. In order the prove
the claim, we need to extend the methods of Barbasch-Moy so that we
cover extended affine (and affine graded in the sense of \cite{L1})
Hecke algebras with unequal parameters.  

\subsection{}\label{sec:consequence} Here is a first important consequence of our methods, indicative of endoscopy. Assume that $\CG$ is quasisplit
quasisimple, and that it splits over an unramified extension of
$\bF$, and $\CK$ is hyperspecial. (When $\CG$ is simple of adjoint type, the
Deligne-Langlands-Lusztig classification for the representations in
$\fk R^{\one}(\CG)$ (and more generally, for unipotent
    representations) is in \cite{L8}.) We would like to relate the unitary representations in
$\fk R^\one(\CG)$ with the unitary dual of certain split endoscopic groups.  The Langlands complex dual group $G$ is  equipped with
an automorphism $\tau$ of the root datum of $G$, defined by the inner
class of $\CG$, see section \ref{sec:quasi}. If $\CG$ is in fact split, then
$\tau$ is trivial. It is well-known that
$W(\CG,A)$-conjugacy classes of unramified complex characters of $M$ are
in one-to-one correspondence with $\tau$-twisted semisimple conjugacy
classes in $G.$ In this correspondence, if $X$ is a subquotient of a
principal series induced from an unramified character $\chi$, we refer
to the corresponding twisted semisimple conjugacy class in $G$ as the
infinitesimal character of $X$. We say that $X$ has real infinitesimal
character if the corresponding semisimple class is hyperbolic modulo
the center.

Fix a semisimple elliptic element $s_e\in G.$ Let $G(s_e\tau)$ denote
the centralizer of $s_e$ in $G$ under $\tau$-twisted conjugacy, a
reductive group.  When
$G$ is simply-connected (so $\CG$ is adjoint), $G(s_e\tau)$  is a connected
group, but not in general. Let $\CG(s_e\tau)$ denote the split
$\bF$-form of a (possibly disconnected) group dual to
$G(s_e\tau).$ There is a natural one-to-one
correspondence between infinitesimal characters of $\CG$ with elliptic part $\tau$-conjugate to $s_e$ and real infinitesimal characters of
$\CG(s_e\tau)$. Moreover, we find that the affine graded at $s_e$
Hecke algebras for $\CG$ and the affine graded at the identity Hecke algebra $\CG(s_e\tau)$ are
naturally isomorphic. The methods of this paper then imply the following result.

\begin{theorem}\label{t:1.3} Let $\CG$ be a quasisplit quasisimple group which
  splits over an unramified extension of $\bF.$ Fix $s_e$ an elliptic element in the dual complex group $G$. 
There is a natural one-to-one correspondence between irreducible
representations in $\fk
R^{\one}(\CG)$ whose infinitesimal character has elliptic part $\tau$-conjugate with $s_e$, and representations in $\fk R^{\one}(\CG(s_e\tau))$ with real infinitesimal character, such that the hermitian and
unitary modules correspond, respectively.  
\end{theorem}

For example, if $s_e=1$ and $\CG$ is the quasisplit form of the unitary group
$PSU(2n)$ or $PSU(2n+1)$, then $\CG(\tau)$  is the split form of
$SO(2n+1)$ or $Sp(2n)$, respectively. In particular, we
obtain a correspondence between the
spherical unitary duals with real infinitesimal character of $\CG$ and
$\CG(\tau).$ This identification of spherical unitary duals (but by
different methods) is also
known to hold for the pairs of classical real groups $(U(n,n),
SO(n+1,n))$ and $(U(n+1,n),Sp(2n,\bR))$, see \cite{Ba}. 

If $\CG$ does not split over an unramified extension of $\bF,$ it is likely that one may
apply the method used in \cite[section 10.13]{L8}. As it is shown
there, for every such $\CG$, there exists a different group $\CG'$
which splits over an unramified extension of $\bF$, with the property that
the Iwahori-Hecke algebra $\CH(\CG,\one_\CI)$ can be identified with
the Iwahori-Hecke algebra  $\CH(\CG',\one_{\CI'})$. Therefore the
categories  $\fk R^\one(\CG)$ and $\fk R^\one(\CG')$ are
equivalent. (See \cite[page 280]{L8} for
the  list of pairs $(\CG,\CG')$.)

\subsection{}
The second example is when $^0\!\chi$ is a nontrivial character of
$^0\!M$. We rely on the theory of types results of \cite{Ro} for ramified
principal series, so we need to assume that $\CG$ is split, and have
certain restrictions on the characteristic of $\bF$.  In this case too, our methods imply a correspondence of
endoscopic type (see section \ref{sec:6.2}). To $^0\!\chi$ one attaches
a semisimple element $\widehat{^0\!\chi}$ in the Langlands dual $G$. Let
$C_G(\widehat{^0\!\chi})$ be the centralizer of $\widehat{^0\!\chi}$ in $G.$
This is a possibly disconnected reductive group. We define a dual
split group $\CG'(^0\!\chi)$, the $\bF$-points of a disconnected
reductive group defined over $\bF$ {(section \ref{sec:6.2}).}

\begin{theorem}\label{t:1.4} Let $\CG$ be a split group and $^0\!\chi$ a nontrivial character of $^0\!M.$
In the isomorphism of categories $\fk R^{^0\!\chi}(\CG)$ and $\fk
R^{\one}(\CG'(^0\!\chi))$ (from \cite{Ro}, see section \ref{sec:6.2}), the hermitian and unitary
representations correspond, respectively.
\end{theorem}

\medskip

\noindent{\bf Acknowledgements.} {The authors thank
  M. Solleveld for a pointing out several inaccuracies in a previous
  version of this paper.} This research was supported in part
by the NSF-DMS 0554278 and 0901104 for D.B., NSF-DMS 0968065 and
NSA-AMS 081022 for D.C.

\section{Affine Hecke algebras and graded affine Hecke
  algebras}\label{sec:2} 
In this section we recall the definitions of the affine Hecke algebra,
its graded version, and the relation between their unitary 
duals. We follow \cite{L1} and \cite{BM2}. There are certain
  minor modifications because we need to consider extended Hecke
  algebras.

\subsection{}\label{sec:2.1} Let $\Psi=(X,X^\vee,R,R^\vee)$ be a root
datum. Thus {$X,X^\vee$ are two
$\bZ-$lattices with a perfect pairing $\langle\ ,\ \rangle:X\times
X^\vee\to \bZ$,} the subsets $R\subset X\setminus\{0\}$ and $R^\vee\subset
X^\vee\setminus\{0\}$ are in bijection $\al\in R\longleftrightarrow \check\al\in
R^\vee,$ satisfying $\langle \al,\check\al\rangle=2$. For
every $\al\in R,$ the reflections $s_\al:X\to X,$ $s_\al(x)=x-\langle
x,\check\al\rangle\al,$ and $s_\cha:X^\vee\to X^\vee,$
$s_\cha(y)=y-\langle \al,y\rangle\check\al$, leave  $R$ and $R^\vee$ stable
respectively. Let $W$ be the finite Weyl group, \ie the group
generated by the set $\{s_\al:\al\in R\}.$  We fix a choice of positive
roots $R^+,$ with basis $\Pi$ of
simple roots, and let $R^{\vee,+}$, $\Pi^\vee$ be the corresponding
images in $R^\vee$ under $\al\mapsto\check\al.$ We assume that $\Psi$
is reduced (\ie, $\al\in R$ implies $2\al\notin R$). Let
$\ell$ be the length function of $W$ with respect to the basis $\Pi$ of simple roots.

{The connected {complex} 
linear reductive group corresponding to $\Psi$ is denoted $G(\Psi)$ or
just $G$ if there is no danger of confusion. Then
$T:=X^\vee\otimes_\bZ\bC^{\times}$ is a maximal torus in $G,$ and let $B\supset
T$ be the Borel subgroup such that the roots of $T$ in
$B$ are $R^+.$}

A parameter set for $\Psi$ is a pair of functions $(\lambda,\lambda^*)$,
$$
\begin{aligned}
 &\lambda:\Pi\to \bZ_{\ge 0},&\lambda^*:\{\al\in\Pi:\check\al\in 2X^\vee\}\to \bZ_{\ge 0}, 
\end{aligned}
$$
such that $\lambda(\al)=\lambda(\al')$ {and
  $\la^*(\al)=\la^*(\al')$} whenever $\al,\al'$ are $W-$conjugate.

\begin{definition}\label{d:2.1} The affine Hecke algebra
  $\CH^{\lambda,\lambda^*}(\Psi,z)$, or just $\CH(\Psi)$, 
associated to the root datum
  $\Psi$ with parameter set $(\lambda,\lambda^*),$ is the
  associative algebra over $\bC[z,z^{-1}]$ with unit ($z$ is an indeterminate), defined by generators $T_w$, $w\in W$, and $\theta_x$,
  $x\in X$ with relations:
\begin{align}
&(T_{s_\al}+1)(T_{s_\al}-z^{2\lambda(\al)})=0,\text{ for all }\al\in\Pi,\\
&T_wT_{w'}=T_{ww'},\text{ for all }w,w'\in W\text{ such that }\ell(ww')=\ell(w)+\ell(w'),\notag\\
&\theta_x\theta_{x'}=\theta_{x+x'},\text{ for all }x,x'\in X,\\
&\theta_x T_{s_\al}-T_{s_\al}\theta_{s_\al(x)}=(\theta_x-\theta_{s_\al(x)}) \C
  (G(\al)-1),\text{ where } x\in X, \al\in\Pi,\text{ and } \notag\\
&\C G(\al)=\begin{cases} 
\frac{\theta_\al z^{2\lambda(\al)}-1}{\theta_\al-1}, &\text{ if }
\check\al\notin 2X^\vee,\\
    \frac{(\theta_\al z^{\lambda(\al)+\lambda^*(\al)}-1)(\theta_\al
      z^{\lambda(\al)-\lambda^*(\al)}+1)}{\theta_{2\al}-1}, &\text{ if
    }\check\al\in 2X^\vee.
 \end{cases}
\end{align}
{Let $\Gamma$ be a finite group endowed with a homomorphism
$\Gamma\longrightarrow \Aut(G,B,T),$ satisfying the property that
$\la(\gamma(\al))=\la(\al)$ and $\la^*(\gamma(\al))=\la^*(\al),$ for
all $\gamma\in\Gamma$. Then we can  form the extended affine Hecke algebra 
\begin{equation}\C H'(\Psi):=\C
H(\Psi)\rtimes\Gamma,\end{equation}
by adding the generators $\{T_\gamma\}_{\gamma\in\Gamma}$ and relations
\begin{equation*}
  \begin{aligned}
&T_\gamma T_w=T_{\gamma(w)}T_\gamma,\ T_\gamma
T_{\gamma'}=T_{\gamma\gamma'},T_\gamma\theta_x=\theta_{\gamma(x)}T_\gamma,\\
&\quad\text{ for } \gamma\in\Gamma,\ w\in W,\ x\in X. 
  \end{aligned}
\end{equation*}
}
\end{definition}

\begin{remarks}

\noindent (1) {The case 
$\Pi^\vee\cap 2X^\vee\neq\emptyset$ can occur only if $R$ has a factor of
type $B$.}
 
\noindent (2)  Since {
  $\theta_x-\theta_{s_\al(x)}=\theta_{x}(1- \theta_{-\al}^n),$}
  where $n=\langle x,\check\al\rangle$,  the denominator of
  $\C G(\al)$ actually divides $\theta_x-\theta_{s_\al(x)}$.

\noindent (3) If $\lambda(\al)=c$ for all $\al\in \Pi$, and
$\lambda^*(\al)=c$, for all 
  $\cha\in 2X^\vee,$ we say that
  $\C H(\Psi)$ is a Hecke algebra with equal parameters. For example, assume
  that $c=1,$ and $z$ acts by $\sqrt q.$ If $\Psi$ corresponds to
  $SL(2,\bC),$ the algebra is generated by $T:=T_{s_\al}$ and
  $\theta:=\theta_{\frac 12\al}$, where $\al$ is the unique positive root,
  subject to  
\begin{align}
&(T+1)(T-q)=0;\\\notag
&\theta T-T\theta^{-1}=(q-1)\theta.
\end{align}
On the other hand, if $\Psi$ corresponds to $PGL(2,\bC),$ then the
generators are $T:=T_{s_\al}$ and $\theta':=\theta_\al$ subject to
\begin{align}
&(T+1)(T-q)=0;\\\notag
&\theta'T-T(\theta')^{-1}=(q-1)(1+\theta').
\end{align}

\end{remarks}

\subsection{}\label{sec:2.2} Recall
$T=X^\vee\otimes_\bZ\bC^\times$. Then 
\[X=\Hom(T,\bC^\times),\qquad X^\vee=\Hom(\bC^\times, T).
\] 
Let $\CA=\CA(\Psi)$ be the algebra of regular functions on {$\bC^\times\times
  T.$}  It can be identified with the abelian $\bC[z,z^{-1}]-$subalgebra of
$\CH(\Psi)$ generated by $\{\theta_x:x\in X\}$, where for
  every $x\in X,$  $\theta_x:T\to \bC^\times$ is defined by
\begin{equation}
\theta_x(y\otimes\zeta)=\zeta^{\langle x,y\rangle},\
y\in X^\vee,~\zeta\in\bC^\times.   
\end{equation}
If we denote by $\C H_W$ the $\bC[z,z^{-1}]-$subalgebra generated by
$\{T_w:w\in W\}$, then 
\begin{equation}\label{eq:2.2.1}
\CH(\Psi)=\C H_W\otimes_{\bC[z,z^{-1}]}\CA
\end{equation}
 as a $\bC[z,z^{-1}]-$module.
An important fact is that as algebras,
$$
\C H_W\cong\bC[W]\otimes_\bC\bC[z,z^{-1}].
$$ 
In the case of an extended algebra, set $W':=W\rtimes\Gamma$
{and $\CH_{W'}:=\CH_W\rtimes\Gamma$. Then 
\begin{equation}\label{eq:2.2.1a}
\CH'(\Psi)=\C H_{W'}\otimes_{\bC[z,z^{-1}]}\CA
\end{equation} }
as a  $\bC[z,z^{-1}]-$module. 

\begin{theorem}[Bernstein-Lusztig, {\cite[Proposition 3.11]{L1}}]\label{t:2.2} The center of
  $\CH(\Psi)$ is $\C Z=\CA^W,$ i.e., the
  $W-$invariants in $\CA.$ 
Similarly, the center of $\C H'(\Psi)$ is $\C A^{W'}.$ 
\end{theorem}

Let $\text{mod} \CH(\Psi)$ or $\text{mod} \CH'(\Psi)$ denote the categories of
finite dimensional Hecke algebra modules. {By Schur's lemma, every
irreducible module $(\pi,V)$ has a \textit{central character}, \ie there
is a homomorphism $\chi:\C Z\longrightarrow\bC$ such that
$\pi(z)v=\chi(z)v$ for every $v\in V$ and $z\in \C Z.$} 
By Theorem \ref{t:2.2}, the central characters correspond
to $W-$conjugacy (respectively $W'-$conjugacy) classes $(z_0,s)\in
\bC^\times\times T$. Then, we have:
\begin{equation}
\text{mod} \CH(\Psi)=\bigsqcup_{(z_0,s)\in
  \bC^\times\times W\backslash T} 
\text{mod}_{(z_0,s)} \CH(\Psi),
\end{equation}
where $\text{mod}_{(z_0,s)} \CH(\Psi)$ is the
subcategory of modules with central character (corresponding to) $(z_0,s).$ 
Let $\text{Irr}_{(z_0,s)}\CH(\Psi)$ be the set of isomorphism classes of
simple objects in this category. One has the similar definitions for
$\CH'(\Psi).$  Throughout the section, 
we will assume that $z_0$ is a fixed number in $\bR_{>1}.$

\subsection{}\label{sec:2.3} { Fix a $W-$orbit $\C O$ of
  an element $\sigma\in T,$} and denote by $\CO'$ the
$W'-$orbit of $\sigma$. Then
$\CO'=(\CO_1=\CO)\sqcup\CO_2\sqcup\dots\sqcup\CO_m$ where each
$\CO_i$ is a $W-$orbit. Following \cite[section 4]{L1} and
  \cite[section 3]{BM2}, define the decreasing chain of ideals 
$\C I^k$, $k\ge 1,$ in $\CA$ as 
{
\begin{equation}
\C I:=\{f\text{ regular function on }\bC^\times\times T:
f(1,\sigma)=0,\forall\sig\in \C O\},
\end{equation}
and $\C I^k$ the ideal of functions vanishing on $\C O$ to at least
order $k.$} Let $\wti {\C I}^k:={\C I}^k\C H=\C H\C I^k$, $k\ge 1,$ be
the chain of ideals in $\CH(\Psi)$, generated by the $\C I^k$'s.
Similarly define $(\C I')^k,$ and $\wti{(\C I')^k}$ in
  $\CH'(\Psi)$ {using the orbit $\CO'.$} 

\begin{definition}\label{d:2.3}
The affine graded Hecke algebra {$\bH_\CO(\Psi)$} is the
associated graded object to the chain of ideals $\dots\supset \wti {\C
  I}^k\supset\dots$ in $\C H(\Psi).$ Similarly 
$\bH_{\CO'}(\Psi)$ is {the  graded object in $\CH'(\Psi)$ for the filtration 
$\dots\supset \wti {(\C I')}^k\supset\dots$ }
\end{definition}

Let $\fk t=X^\vee\otimes_\bZ \bC$ be the Lie algebra of $T,$ and let
$\fk t^*=X\otimes_\bZ\bC$ be the dual space. Extend the pairing
$\langle\ ,\ \rangle$ to $\fk t^*\times \fk t.$ Let $\bA$
be the algebra of regular functions on $\bC\oplus \fk t.$ Note that
$\bA$ can be identified with $\bC[r]\otimes_\bC S(\fk t^*),$ where
$S(~)$ denotes the symmetric algebra, and $r$ is an indeterminate. In
the following $\delta$ denotes the delta function.

{
\begin{theorem}[\cite{L1,BM2}]\label{t:2.3} The graded Hecke algebra
$\bH_{\CO}(\Psi)$ is a $\bC[r]-$algebra generated by
$\{t_w:w\in W\}$,   $S(\fk t^*),$
and a set of orthogonal   idempotents   $\{E_{\sig}:\sig\in\C O\}$  
subject to the relations
\begin{align}
&t_w\cdot t_{w'}=t_{ww'},\ w,w'\in W,\quad
\label{eq:2.3.2}\\
&\sum_{\sigma\in \C O}E_{\sigma}=1, \quad
  E_{\sigma'}E_{\sigma''}=\delta_{\sigma',\sigma''} E_{\sigma'},\quad E_{\sigma}t_{s_\al}=t_{s_\al}E_{s_\al\sigma},\label{eq:2.3.3}\\
&\omega\cdot t_{s_\al}-t_{s_\al}\cdot s_\al(\omega)=r 
  g(\al)\langle\omega,\check\al \rangle, \text{ where }\al\in \Pi,
  \omega\in \fk t^*,\label{eq:2.3.4}\\
& g(\al)=\sum_{\sigma\in\CO}
  E_{\sigma}\mu_{\sigma}(\al), \text{ and }\label{eq:2.3.5}\\
&\mu_{\sigma}(\al)=\left\{ \begin{matrix} 0, &\text{ if } s_\al\sigma\neq
    \sigma,\\
                                  2\lambda(\al), &\text{ if
                                  }s_\al\sigma=\sigma,\ \  \check\al\notin
                                  2X^\vee,\\
           \lambda(\al)+\lambda^*(\al)\theta_{-\al}(\sigma),   &\text{ if
                                  }s_\al\sigma=\sigma,\ \  \check\al\in
                                  2X^\vee.                   
\\ \end{matrix}  \right.\label{eq:2.3.6}
\end{align} 
\end{theorem}
}

Notice that if $s_\al\sigma=\sigma,$ then
$\theta_\al(\sigma)=\theta_{-\al}(\sigma)$, or equivalently,
$\theta_\al(\sigma)\in\{\pm 1\}.$ This implies that for the parameters
$\mu_\sigma(\al)$ we have $\mu_\sigma(\al)\in\{0,2\lambda(\al),
\lambda(\al)-\lambda^*(\al)\}$, for every root $\al\in \Pi.$  In
particular, in the case of equal parameters Hecke algebra, the only
possibilities are $\mu_\sigma(\al)\in \{0,2\lambda(\al)\}.$

\begin{remark}\label{r:2.3} {An important special case is when
 $\C O$ is formed of a single ($W-$invariant) element $\sig.$ 
{Then there is only one
idempotent  generator $E_{\sig}=1$, so it   is suppressed from the notation.}
The algebra $\bH_\CO(\Psi)$ is generated by $\{t_w:w\in W\}$
and $S(\fk t^*)$ subject to the commutation relation
\begin{equation}\label{eq:2.3.7}
\omega\cdot t_{s_\al}-t_{s_\al}\cdot s_\al(\omega)=r
\mu_\sigma(\al)\langle\omega,\check\al\rangle, \quad \al\in
\Pi,\ \omega\in \fk t^*,
\end{equation}
where 
\begin{equation}\label{eq:mu}
\mu_{\sigma}(\al)=\left\{ \begin{matrix} 
                                  2\lambda(\al), &\text{ if
                                  }  \check\al\notin
                                  2X^\vee,\\
           \lambda(\al)+\lambda^*(\al)\theta_{-\al}(\sigma),   &\text{ if
                                  } \check\al\in
                                  2X^\vee.                   
\\ \end{matrix}  \right.
\end{equation}
Still assuming that $\sigma$ is $W-$invariant, we have  
$\theta_{\al}(\sigma)\in\{\pm 1\},$ for all $\al\in \Pi$. If in fact
$\sigma$ is in the center of group $G(\Psi)$, then  $\theta_{\al}(\sig)=1,$
for all $\al\in\Pi.$ 
}
\end{remark}

\begin{notation}{We will use the notation   $\bH_{\mu_\sig}$ for the
    graded Hecke
  algebra in the
  particular case 
  defined by equations (\ref{eq:2.3.7}) and (\ref{eq:mu}).}
\end{notation}

\begin{example} Let $\Psi$ be the root datum for $PGL(2,\bC),$
in the equal parameter case, and $\Gamma=\{1\}$. We present three cases:

\noindent (1) $\sigma=\left(\begin{matrix}
  1&0\\0&1\end{matrix}\right)$. This is clearly $W-$invariant. Then $\bH_{\mu_\sigma}$ is generated by $t=t_{s_\al}$
  and $\omega$ subject to 
  \begin{equation}t^2=1,\qquad t\om +\om t=2r\lambda(\al).\end{equation}
\noindent (2)  $\sigma=\left(\begin{matrix}
    i&0\\0&-i\end{matrix}\right)$. Since we are in $PGL(2,\bC),$ the
    element $\sigma$ is $W-$invariant. Note that $\theta_\al(\sig)=-1,$ and so $\bH_{\mu_\sigma}$
    is generated by $t$ and $\omega$ subject to
\begin{equation}t^2=1,\qquad t\om +\om t=0.\end{equation}
\noindent (3)
$\sigma=\left(\begin{matrix}\zeta&0\\ 0&\zeta^{-1}\end{matrix}\right)$,
$\zeta^2\ne \pm 1.$ Then $W\cdot \sigma=\{\sigma,\sigma^{-1}\}.$  The algebra $\bH_{\CO}$ is generated
by $E_\sigma,\ E_{\sigma^{-1}},\ t,\ \om$ satisfying the following  relations:
  \begin{align}
 &E_\sigma^2=E_\sigma,\ E_{\sigma^{-1}}^2=E_{\sigma^{-1}},\
    E_\sigma\cdot E_{\sigma^{-1}}=0,\ E_\sigma +E_{\sigma^{-1}}=1,\notag\\
 &tE_\sigma=E_{\sigma^{-1}}t,\\\notag
&t^2=1,\qquad t\om +\om t=0.\qed
  \end{align}
\end{example}

{We can think of $\bH_\CO(\Psi)$ as the associative algebra defined by
the relations (\ref{eq:2.3.2})-(\ref{eq:2.3.6}). In particular if there
is a homomorphism $\Gamma\longrightarrow \Aut(G,B,T),$ we can define 
\begin{equation}
  \label{eq:egh}
\bH'_\CO(\Psi):=  \bH_\CO(\Psi)\rtimes\Gamma
\end{equation}
as the associative algebra generated by
$\{t_\gamma\}_{\gamma\in\Gamma}$ and the generators of $\bH_\CO$
satisfying (\ref{eq:2.3.2})-(\ref{eq:2.3.6}), and in addition
\begin{equation}  \label{eq:eghrel}
  \begin{aligned}
&t_\gamma t_w=t_{\gamma(w)}t_\gamma,\ t_\gamma
t_{\gamma'}=t_{\gamma\gamma'},&&\gamma,\gamma'\in\Gamma,\\
&t_\gamma\om=\gamma(\om)t_\gamma,&&
w\in W,\omega\in \fk t^*.    
  \end{aligned}
\end{equation}

\begin{corollary} {There is a natural identification
\begin{equation*}
\bH_{\CO'}(\Psi)=\bH'_\CO(\Psi)(=\bH_\CO(\Psi)\rtimes\Gamma),
\end{equation*}
where $\bH_{\CO'}(\Psi),\bH_\CO(\Psi)$ are as in Definition
\ref{d:2.3}, and $\bH'_\CO(\Psi)$ is as in (\ref{eq:egh}).}
\end{corollary}
\begin{proof}
By (3) of Proposition 3.2 of \cite{BM2},
$\bH_{\CO'}=\bigoplus_{i=1}^m\bH_{\CO_i}$; recall that
$\{\CO_i\}$ is the $W-$orbit partition of the $W'-$orbit $\CO'.$ Each
$T_\gamma$ induces (by grading) an algebra isomorphism 
\begin{equation}
  \label{eq:eghcor1}
  t_\gamma:\bH_{\CO_i}\longrightarrow \bH_{\gamma\CO_i=\CO_j}
\end{equation}
and therefore an automorphism
\begin{equation}
  \label{eq:eghcor2}
  t_\gamma:\bH_{\CO'}=\bigoplus\bH_{\CO_i}\longrightarrow\bigoplus
  \bH_{\gamma\CO_i=\CO_j}=\bH_{\CO'}
\end{equation}
satisfying the required relations. We omit further details.
\end{proof}
}

{
Define 
\begin{equation}\label{eq:2.3.16}
C_{W'}(\sigma):=\{w\in W': w\sigma=\sigma\}.
\end{equation}  
Then $W'\cdot\sigma=\{w_j\cdot \sigma: 1\le j\le n\}$, where
$\{w_1=1,w_2,\dots,w_n\}$ are coset representatives for
$W'/C_{W'}(\sigma).$ Then $\{E_{\tau}\}=\{E_{w_j\sigma}: 1\le j\le
n\},$ and from Theorem \ref{p:2.3}, 
$\bH_\CO'=\bC[W']\otimes(\C E'\otimes \bA)$, as a $\bC[r]-$vector
space, where $\C E'$ is the algebra generated by the
$E_{w_j\cdot\sigma}$'s with $1\le j\le n.$}

{
\begin{proposition}[{\cite[Proposition 4.5]{L1}}]\label{p:2.3}
The center of $\bH_\CO$ is $Z=(\C E\otimes\bA)^{W}.$
The center of $\bH_\CO'$ is  $Z'=(\C E'\otimes\bA)^{W'}.$ 
\end{proposition}
}
It follows that the central characters of $\bH_\CO'$ are
parameterized by $C_{W'}(\sigma)-$orbits in $\bC\oplus \fk t.$ Similarly
to the last paragraph in section \ref{sec:2.2}, define the category
$\text{mod}_{(r_0,x)}\bH_\sigma',$ where $r$ acts by $r_0>0$ and
$x\in \fk t.$

\subsection{}\label{sec:2.4} 
We describe the structure of $\bH_{\CO'}(\Psi)=\bH'_\CO(\Psi)$
in more detail. Fix a $\sig\in\CO'\subset T.$  We define a  root datum
$\Psi_\sigma=(X,R_\sigma,X^\vee,R_\sigma^\vee)$ with positive roots
$R_\sigma^+$, defined as follows:
\begin{align}
&R_\sigma=\left\{\quad \al\in R: \theta_\al(\sigma)=\left\{\begin{matrix}1,
    &\text{ if }\check\al\notin 2X^\vee,\\ \pm 1, &\text{ if
    }\check\al\in 2X^\vee  \end{matrix}\right.\quad \right\},\\
&R_\sigma^+=R_\sigma\cap R^+,\\
&R_\sigma^\vee=\{\check\al\in R^\vee: \al\in R_\sigma\}. 
\end{align}
{Note that $\Psi_\sigma$ is the root datum for $(C_G(\sigma)_0,T)$,
  with positive roots $R_\sigma^+$ with respect to the Borel subgroup $C_G(\sigma)_0\cap B.$} 
{Define
\begin{equation}\label{2.4.4}
\Gamma_\sigma:=\{w\in C_{W'}(\sigma): w(R_\sigma^+)=R_\sigma^+\}.
\end{equation}
There is a group homomorphism 
$\Gamma_\sigma\longrightarrow \Aut(C_G(\sig)_0,C_G(\sig)_0\cap B,T)$
such that $\mu_\sig(\gamma\al)=\mu_\sig(\al)$, so the
extended  Hecke algebra
\begin{equation}\label{eq:2.4.5}
\bH_{\mu_\sigma}'(\Psi_\sigma):=\bH_{\mu_\sig}(\Psi_\sigma)\rtimes \Gamma_\sigma
\end{equation}
is well defined.}

\medskip
Recall the coset representatives  $\{w_1,\dots, w_n\}$  for
$W'/C_{W'}(\sigma)$ from the paragraph after (\ref{eq:2.3.16}). Set
\begin{equation}
E_{i,j}=t_{w_i^{-1}w_j}E_{w_j\sigma}=E_{w_i\sigma}t_{w_i^{-1}w_j},
\text{ for all } 1\le i,j\le n,
\end{equation}
and let $\C M_n$ be the matrix algebra with basis $\{E_{i,j}\}$.

\begin{theorem}[\cite{L1}] 
There is a natural algebra isomorphism
$$
\bH_\CO'(\Psi)\cong\C M_n\otimes_\bC
\bH_{\mu_\sigma}'(\Psi_\sigma)=\C M_n\otimes_\bC
(\bH_{\mu_\sig}(\Psi_\sigma)\rtimes \Gamma_\sigma).
$$ 
\end{theorem}

Since the only irreducible representation of $\C M_n$ is the
$n-$dimensional standard representation, one obtains immediately 
the equivalences of categories:
\begin{equation}\label{eq:2.4.7}
\text{mod}_{(r_0,x)} \bH_{\CO}'(\Psi)\cong \text{mod}_{(r_0,x)}
\bH_{\mu_\sigma}'(\Psi_\sigma). 
\end{equation}

\begin{remarks} 

\noindent (1) When $\Gamma=\{1\}$ and  $X^\vee$ is
generated by $R^\vee$, that is, when $\Psi$ is of simply connected
type, or more generally, if $X^\vee$ is generated by $R^\vee\cup \frac
12 R^\vee$ (which includes the case of factors of type $B$ as well),
then $C_W(\sigma)\subset W_\sigma,$ and so $\Gamma_\sigma=\{1\},$ for
every $\sigma\in T.$ In this case,
$\bH_\CO'(\Psi_\sigma)=\bH_\CO(\Psi_\sigma),$ and there is no
need to consider the extended graded Hecke algebras (\ref{eq:2.4.5}). 

\noindent (2) When $\sigma$ is $W'-$invariant, then $n=1,$ and so
$\bH_\CO'(\Psi)\cong
\bH_{\mu_\sigma}'(\Psi_\sigma)=\bH_{\mu_\sigma}(\Psi_\sigma)\rtimes \Gamma_\sigma.$

\end{remarks}

\subsection{}\label{sec:2.5} In this section, we discuss the relation
between $\C H(\Psi)$ and $\bH_\CO(\Psi).$ We need some
definitions first. 

The torus $T=X^\vee\otimes_\bZ\bC^\times$ admits a polar decomposition
$T=T_e\times T_h,$ where $T_e=X^\vee\otimes_\bZ S^1,$ and
$T_h=X^\vee\otimes_\bZ \bR_{>0}.$ Consequently, every $s\in T$
decomposes uniquely into $s=s_e\cdot s_h,$ with $s_e\in T_e$ and
$s_h\in T_h.$ We call an element $s_e\in T_e$ elliptic, and an element
$s_h\in T_h$ hyperbolic. Similarly, $\fk t=X^\vee\otimes_\bZ\bC$
admits the decomposition $\fk t=\fk t_{i\bR}\oplus\fk t_\bR$ into an
imaginary part $\fk t_{i\bR}=X^\vee\otimes_\bZ i\bR$ and a real part $\fk
t_\bR=X^\vee\otimes_\bZ \bR.$ 

{
We need to define certain completions of the Hecke algebras. The
algebras $\bC[r],\ \mathscr{S},$ and $\bC[r]\otimes\mathscr{S}$ consist
of polynomial functions on $\bC,\ \fk t$ and {$\C
  M:=\bC\oplus\fk t$}, respectively. Let 
$\wht\bC[r],\ \wht{\mathscr{S}},$ and $\wht\bC[r]\otimes\wht{\mathscr{S}}$ 
be the  corresponding algebras of holomorphic functions. Let $\mathscr{K}$ and
$\wht{\mathscr{K}}$ be the fields of rational and meromorphic functions
on $\C M.$ Finally set 
$\wht\bA:=\bA\otimes_{\mathscr{S}}\wht{\mathscr{S}}\subset\wht{\mathscr K},$ 
and
\begin{align}
  \label{eq:compl}
&\wht{\bH}_\CO:=\bC [W]\otimes(\wht\bC[r]\otimes\wht \bA),\notag\\
&\wht\bH_\CO({\mathscr{K}}):=\bC [W]\otimes(\C E \otimes\mathscr{K})\supset \wht\bH_\CO,\notag\\
&\wht{\bH}_\CO(\wht{\mathscr{K}}):=\bC [W]\otimes(\C E\otimes\wht{\mathscr{K}})
\supset \bH_\CO(\mathscr{K}),\wht\bH_\CO.\notag
\end{align}
We make the analogous definitions for $\bH'$.
\begin{theorem}[{\cite[section 5.2]{L1} and \cite[Theorem 3.5]{BM2}}]\label{t:2.5} 
The map
\begin{align*}&\iota:\bC[W]\rtimes(\C
E\otimes\wht{\mathscr{K}})\longrightarrow\wht\bH_\CO(\wht{\mathscr{K}}),\text{
  defined by}\\
&\iota(E_\sig)=E_\sig,\quad \iota(f)=f,\ f\in\wht{\mathscr{K}},\\
&\iota(t_\al)=(t_\al+1)(\sum_{\sigma\in\CO}
{g_\sig(\al)^{-1}}E_\sig) -1,
\end{align*}
where
{\begin{equation}\label{eq:gal}
g_\sig(\al)=1+\mu_\sig(\al)\al^{-1}\in\wht{\mathscr{K}},
\end{equation}}
is an algebra isomorphism. Similarly for extended algebras, we have
the analogous isomorphism $\iota':\bC[W']\rtimes(\C
E\otimes\wht{\mathscr{K}})\longrightarrow\wht\bH'_\CO(\wht{\mathscr{K}}),$ with $\iota'(\gamma)=\gamma.$
\end{theorem}

To every character $\chi$ of $\C Z$ (the center of $\CH=\CH_W\otimes\CA$),  there corresponds a maximal
ideal $\C J_\chi=\{z\in \C Z: \chi(z)=0\}$ of $\C Z.$ Define the
quotients
\begin{equation}
\CA_\chi=\CA/\CA\cdot\C J_\chi,\quad
\CH(\Psi)_\chi=\CH(\Psi)/\CH(\Psi)\cdot\C J_\chi.
\end{equation}
Similarly, consider the ideal $\C I_{\overline\chi}$ for every character
$\overline\chi$ of $Z$ in $\bH_\CO$, and define the analogous
quotients. 
Then
\begin{align}
  \label{eq:2.5.4}
&\bA_{\ovl\chi}:=\bA/(\bA\cdot\C I_{\ovl\chi})=
\wht\bA/(\bA\cdot\wht{\C I}_{\ovl\chi})=\wht\bA_{\ovl\chi},\notag\\
&\bH_{\ovl\chi}:=\bH_\CO/(\bH_\CO\cdot\C I_{\ovl\chi})=
\wht\bH_\CO/(\wht\bH_\CO\cdot\C I_{\ovl\chi})=\wht\bH_{\ovl\chi}.
\end{align}
The similar definitions and formulas hold for $\CH'$ and $\bH'_\CO.$ 
}

The map 
\begin{equation}
\tau:\bC\oplus\fk t_\bR\to \bC^\times\times T,\quad (r_0,\nu)\mapsto
(z_0,s)=(e^{r_0},\sigma\cdot e^\nu)
\end{equation}
is $C_W(\sigma)-$invariant. It matches the central characters
\begin{equation}\label{eq:2.5.3} 
\tau:~\overline\chi=C_W(\sigma)\cdot (r_0,\nu)\longleftrightarrow
\chi=W\cdot (z_0,s).
\end{equation}
Moreover, $\tau$ is a bijection onto the central characters of $\C
H$ with elliptic part in $\CO.$ Similarly for extended
algebras, {we have a matching $\tau': \overline\chi'=C_{W'}(\sigma)\cdot (r_0,\nu)\longleftrightarrow
\chi=W'\cdot (z_0,s)$ which is a bijection onto central characters of
$\CH'$ with elliptic part in $\CO'.$} 

{
\begin{proposition}[{\cite[Proposition 4.1]{BM2}}]\label{p:2.5}
The map $\phi:\C A[z,z^{-1}]\longrightarrow\wht\bC[r]\otimes\wht\bA$ defined by
\begin{equation}
\begin{aligned}
&\phi(z)=e^r,\\
&\phi(\theta_x)=\sum_{\sigma\in\CO}\theta_x(\sig)E_\sig e^x,\qquad x\in
X,
\end{aligned}
\end{equation}
is a $\bC-$algebra homomorphism which maps $\C J_\chi$ to $\C
I_{\ovl\chi}$ and defines by passage to the quotients an isomorphism
between $\C A_\chi$ and $\bA_{\ovl\chi}.$

\noindent The map $\Phi:\C H\longrightarrow\wht\bH_\CO(\wht{\mathscr{K}})$ 
defined by 
\begin{equation}
  \label{eq:2.5.5}
\begin{aligned}
&\Phi(a)=\phi(a),\\
&\Phi(T_\al+1)=\sum_{\sigma\in\CO} E_\sig(t_\al+1){\phi(\C
  G_\al) g_\sig(\al)^{-1}}    
\end{aligned}
\end{equation}
 {with $\C G_\al$ as in Definition \ref{d:2.1}, and $g_\sig(\al)$ as in
 (\ref{eq:gal}),} induces an isomorphism
between $\C H_\chi$ and $(\bH_\CO)_{\ovl\chi}.$ 
\end{proposition}
The map $\Phi$ depends on {$(r_0,\nu)\in\C M_\bR(:=\bC\oplus\fk
  t_\bR)$. We write $\Phi_{(r_0,\nu)}$ when we want to emphasize this dependence.}
\begin{theorem}[{\cite[Theorem 4.3]{BM2}}]\label{t:2.5a} Assume
  $\sigma\in T_e.$ Let   $\chi=W\cdot (e^{r_0},\sigma\cdot e^\nu)$,
  $\overline\chi=C_W(\sigma)\cdot   (r_0,\nu)$ be as in
  (\ref{eq:2.5.3}), with $(r_0,\nu)\in   \bC\oplus\fk t_\bR.$ The
  isomorphism 
$$
\Phi_{(r_0,\nu)}:\CH(\Psi)_\chi\xrightarrow{\cong}
  \bH_\CO(\Psi)_{\overline\chi}
$$
 from   (\ref{eq:2.5.4}) is
  analytic in $(r_0,\nu)\in \bC\times\fk t_\bR$.
\end{theorem} 
Both the proposition and the theorem hold with the obvious
modifications for $\CH'$ and $\bH'.$

\subsection{}\label{sec:2.6} The algebras $\CH(\Psi)$ and
$\bH_\CO(\Psi)$ have natural $*-$operations. These are complex
conjugate involutive anti-automorphisms
defined on generators  as follows, \cite[section 5]{BM2}.

For $\CH(\Psi),$ the generators are $z$, $T_w$, $w\in W$, and $\theta_x,$
$x\in X$ (Definition \ref{d:2.1}), and we set:
\begin{equation}
z^*=z,\quad T_w^*=T_{w^{-1}},\quad \theta_x^*=T_{w_0}\theta_{-w_0x}T_{w_0}^{-1},
\end{equation}
where $w_0$ is the longest Weyl group element. In fact, $\CH(\Psi)$
has a structure of normalized Hilbert algebra given by the inner
product $[x,y]=\epsilon(y^*x)$, where $\epsilon:\CH(\Psi)\to \bC$ is
defined by $\epsilon(T_w)=\delta_{w,1}.$
{For $\C H'$, $*$ acts in addition 
by
\begin{equation}
  \label{eq:2.6.1}
  T_\gamma^*:=T_{\gamma^{-1}},
\end{equation}
and set $\epsilon(T_\gamma)=0.$

}
For $\bH_\CO(\Psi)$, recall that the generators are $r,$ $t_w,$ $w\in
W,$ $\omega\in\fk t^*,$ $E_{\sigma'}$, $\sigma'\in W\cdot\sigma.$
The graded $*$-operation:
\begin{align}
&r^*=r,\quad t_w^*=t_{w^{-1}},\quad E_{\sigma'}^*=E_{{\sigma'}^*}\\\notag
&\omega^*=-\omega+r\sum_{\beta\in R^+} t_{s_\beta}\langle
\omega,\check\beta\rangle\sum_{\sigma'\in W\cdot\sigma}
E_{\sigma'}\mu_{\sigma'}(\beta), \quad \omega\in\fk t^*_\bR,
\end{align} 
where 
\begin{equation}
\text{if }s=s_e\cdot s_h\in T=T_e\times T_h,\qquad s^*:=s_e\cdot s_h^{-1}.
\end{equation}
Following \cite{BM2}, we call $s$ such that $s^*\in W\cdot s$
hermitian. Note that every elliptic element is hermitian, and
therefore, under our assumption that $\sigma$ be elliptic, the
$*-$operation on $\bH_\sigma(\Psi)$ is well-defined.
{For $\bH'$, define in addition
  \begin{equation}
    \label{eq:2.6.2}
    t_\gamma^*=t_{\gamma^{-1}}.
  \end{equation}
}
Using these $*-$operations, we define hermitian and unitary modules
for $\CH(\Psi)$ and $\bH_\CO(\Psi)$ as well as for the extended
  algebras $\CH'(\Psi)$ and $\bH_\CO'(\Psi).$

\begin{proposition}[{\cite[Theorem 5.7]{BM2}}]\label{p:2.6} In the setting of Theorem
  \ref{t:2.5a}, the isomorphism $\CH(\Psi)_\chi\cong
  \bH_\CO(\Psi)_{\overline\chi}$ is compatible with the
  $*-$structures. In the equivalences of categories
 \begin{align}\label{eq:2.6.4}
&\text{mod}_\chi\CH(\Psi)\cong
  \text{mod}_{\overline\chi}\bH_\CO(\Psi),\\
&\text{mod}_\chi\CH'(\Psi)\cong
  \text{mod}_{\overline\chi}\bH'_\CO(\Psi),
\end{align}
 the hermitian irreducible
  modules and the unitary irreducible modules correspond,
  respectively. 
\end{proposition}

\begin{corollary}\label{c:2.6a}The equivalence of categories 
\begin{equation}\label{eq:2.6.5}
\C F: \text{mod}_{\overline\chi}\bH_{\mu_\sigma}'(\Psi_\sigma)\cong\text{mod}_\chi\CH'(\Psi)
\end{equation}
given by combining (\ref{eq:2.6.4}) and (\ref{eq:2.4.7}), takes hermitian
irreducible modules to hermitian irreducible modules and
unitary irreducible modules to unitary irreducible modules.
\end{corollary}

\begin{proof}
This is immediate from Proposition \ref{p:2.5}, since tensoring with
the $n-$dimensional standard representation of {$\C M_n$} preserves
irreducible, hermitian, and unitary modules, respectively.
\end{proof}

\noindent{\bf Remark.} Corollary \ref{p:2.6} effectively says that, in order to
compute the unitary dual for {$\CH'(\Psi),$} it is equivalent to compute
the unitary dual with real central character for extended graded Hecke
algebras
$\bH_{\mu_\sig}'(\Psi_\sigma)=\bH_{\mu_\sigma}(\Psi_\sigma)\rtimes
\Gamma_\sigma$,  
for every conjugacy class of $\sigma\in T_e$ (elliptic
semisimple elements). In the case when $\sigma$ is a central element
of $G$, this means that the
unitary dual in $\text{mod}_{(z_0,s)}\C H(\Psi)$ with $s_e=\sig$ is identified
with the unitary dual with real central  character for $\bH_1(\Psi_1).$ \qed

\medskip

For future purposes, we record here how the functor $\C F$ from
(\ref{eq:2.6.5}) behaves with respect to the $W'-$structure. The first part is
Corollary 3.4.(2) in \cite{BM2}, and the second part is immediate from
the remark after Theorem \ref{p:2.3}.

\begin{corollary}[{\cite[Corollary 3.4]{BM2}}]\label{c:2.6} As $\bC[W']-$modules:
\begin{equation}
\C F(\overline V)=\Ind^{W'}_{C_{W'}(\sigma)}(\overline V),
\end{equation}
where $\overline V\in
\text{mod}_{\overline\chi}\bH_{\mu_\sigma}'(\Psi_\sigma)$, $\C
F(\overline V)\in \text{mod}_\chi\CH(\Psi).$

In particular, if $\sigma$ is
$W'-$invariant, then $\C F(\overline V)\cong \overline V$ as $\bC[W']-$modules.
\end{corollary}

\subsection{}\label{sec:2.7} Let $V$ be a
{$\CH'(\Psi)-$module} on which $z$ acts by $z_0\in
\bR_{>1}$. For every $t\in T,$ define a $\CA-$generalized eigenspace
of $V$: 
\begin{equation}
V_t=\{v\in V: \text{ for all } x\in X,\ (x(t)-\theta_x)^kv=0, \text{
  for some } k\ge 0\}.
\end{equation}
We say that $t$ is a weight of $V$ if $V_t\neq\{0\}.$ Let $\Phi(V)$
denote the set of weights of $V.$ We have $V=\oplus_{t\in\Phi(V)}
V_t$ as $\CA$-modules.

\begin{definition}\label{d:2.7}
We say that $V$ is tempered if, for all $x\in X^+:=\{x\in X:\langle
x,\check\al\rangle\ge 0, \text{ for all }\al\in R^+\}$, and all
$t\in\Phi(V),$  $|x(t)|\le 1$ holds.
\end{definition}

Let $\overline V$ be a {
  $\bH_{\mu_\sigma}'(\Psi_\sigma)-$module} on which $r$ acts by 
$r_0>0.$ For every $\nu\in\fk t,$ define a $\bA-$generalized eigenspace
of $\overline V$,
\begin{equation}
\overline V_\nu=\{v\in \overline V: \text{for all }\omega\in S(\fk
t^*),\ (\omega(\nu)-\omega)^kv=0, \text{ for some }k\ge 0\}.
\end{equation}
We say that $\nu$ is a weight of $\overline V$ if $\overline
V_\nu\neq\{0\}.$ Let $\Phi(\overline V)$ denote the set of weights
of $\overline V.$ We have 
$
\overline V=\oplus_{\nu\in\Phi(\overline V)} V_\nu,
$ as $\bA$-modules.

\begin{definition}
We say that $\overline V$ is tempered if, for all $\omega\in
\bA^+:=\{\omega\in S(\fk t^*): \langle\omega,\check\al\rangle\ge
0,\text{ for all }\al\in R^+\}$, and all $\nu\in\Phi(\overline V)$,
we have $\langle\omega,\nu\rangle\le 0.$
\end{definition}

The two notions of tempered are naturally related, see \cite[sections
6 and 7 ]{BM2}, also \cite[Lemmas 4.3 and 4.4]{L8}.

\begin{lemma}\label{l:2.7} In the equivalence of categories from
  Corollary \ref{c:2.6}, the tempered modules correspond. Precisely,
  $\overline V$ is tempered if and only if $V=\C F(\overline V)$ is
  tempered, where $\C F$ is as in (\ref{eq:2.6.5}).
\end{lemma}

\section{Geometric graded Hecke algebras and linear
  independence}\label{sec:3} 

\subsection{}\label{sec:3.1}
{
In this section, we discuss the graded Hecke algebra $\bH=\bH_{\mu_\sigma}(\Psi_\sigma)$ defined by
the relations in (\ref{eq:2.3.7}) and (\ref{eq:mu}) for the case $\sig=1$ (and $\Gamma=\{1\}$). We
recall its definition in this particular case. Let $(X,R,X^\vee,R^\vee)$ be a root datum
for a reduced root system, with finite Weyl group $W,$ and simple roots
$\Pi$. The relations are in terms of the function $\mu$
in (\ref{eq:mu}).  The lattices $X$ and $X^\vee$ are not explicitly needed
for the relations, just $\fk t:=X^\vee\otimes_\bZ\bC$ and $\fk
t^*=X\otimes_\bZ\bC.$ Then
$\bH=\bH^\mu(\fk t^*,R)$ is generated over $\bC[r]$ by
$\{t_w:w\in W\}$ and $\omega \in\fk t^*$, subject to
\begin{align}\label{eq:3.3.1}
&t_w\cdot t_{w'}=t_{ww'},\quad w,w'\in W,\\\notag
&\omega\cdot\omega'=\omega'\cdot\omega,\quad \omega,\omega'\in\fk t^*,\\\notag
&\omega\cdot t_{s_{\al_i}}-t_{s_{\al_i}}\cdot
s_{\al}(\omega)=2r\mu_\al\langle\omega,\check\al \rangle,\quad
\omega\in\fk t^*,~ \al\in\Pi.
\end{align} 
We will consider only Hecke algebras of geometric type, i.e., those
arising by the construction of \cite{L2}. We will recall this next,
but let us record first what the explicit cases are. Assume the root
system is simple. Then there are at most two $W-$conjugacy classes in
$R^+.$ Since $\mu$ is constant on $W-$conjugacy classes,
it is determined by its values $\mu_s:=\mu(\al_s)$ and
$\mu_l:=\mu(\al_l)$, where $\al_s$ is a short simple root, and $\al_l$
is a long simple root. For uniformity of
notation in the table below, we say $\mu_s=\mu_l$ in the simply-laced
case. Only the ratio of these parameters is important, and
there are also obvious isomorphisms between types $B$ and $C$.  The
list is in Table \ref{ta:3.3}.
}
\medskip
\begin{table}[h]\caption{Lusztig's geometric parameters\label{ta:3.3}}
\begin{tabular}{|c|c|}\hline
Type &ratio $\mu_s/\mu_l$\\
\hline
$A_n$ &$1$\\
$B_n$ &$\bZ_{>0},\ \bZ_{>0}+1/2,\ \bZ_{>0}+1/4,\ \bZ_{>0}+3/4$\\
$D_n$ &$1$\\
$E_{6,7,8}$ &$1$\\
$F_4$ &$1,\ 2$\\
$G_2$ &$1,\ 9$\\
\hline
\end{tabular}
\end{table}

\subsection{}\label{sec:3.2} We review briefly the construction and
classification of these
algebras following the work of Lusztig.  Let $G$ be a complex connected
reductive group, with a fixed Borel subgroup $B$,
and maximal torus $H\subset B$. The Lie algebras will be denoted by
the corresponding German letters.

\begin{definition}
A cuspidal triple for $G$ is a triple $(L,\C C,\C L),$ where $L$ is a
Levi subgroup of $G,$ $\C C$ is a nilpotent adjoint $L-$orbit on the Lie
algebra $\fk l$, and $\C L$ is an irreducible $G-$equivariant local
system on $\C C$ which is cuspidal in the sense of \cite{L5}. 
\end{definition}
Let $\fk L(G)$ denote the set of $G-$conjugacy
classes of cuspidal triples for $G$. For example, $(H,0,\triv)\in \fk L(G).$ 
Let us fix $(L,\C C,\C L)\in \fk L(G)$, such that $H\subset L,$
and $P=LU\supset B$ is a parabolic subgroup.  Let
$T$ denote the identity component of the center of $L$. Set
$W=N_G(L)/L$. This is a Coxeter group due to the particular form $L$
must have to allow a cuspidal local system.

\begin{definition}[{\cite[section 2]{L2}}] Let $\bH(L,\C C,\C L):=\bH^\mu(\ft^*,R)$ define a
  graded Hecke   algebra as in (\ref{eq:3.3.1}) where:
\begin{enumerate}
\item[(i)] $\ft$ is the Lie algebra of $T$;
\item[(ii)] $W=N_G(L)/L$;
\item[(iii)] $R$ is the reduced part of the root system given by
  the nonzero weights of  $\text{ad}(\ft)$ on $\fg$; it can be
  identified with the root system of the reductive part of $C_G(x),$ where $x\in \C C$; 
\item[(iv)] $R^+$ is the subset of $R$ for which the
  corresponding weight space lives in $\fk u$;
\item[(v)] the simple roots $\Pi=\{\al_i:i\in I\}$ correspond to the Levi
  subgroups $L_i$ containing $L$ maximally: $\al_i$ is the unique
  element in $R^+$ which is trivial on the center of $\fk l_i$;
\item[(vi)] for every simple $\al_i$, $\mu_{\al_i}\ge 2$ is defined to be
  the smallest integer  such that
\begin{equation}
\text{ad}(x)^{\mu_{\al_i}-1}:\fk l_i\cap\fk u\to \fk l_i\cap\fk u
\text{ is zero.} 
\end{equation} 
\end{enumerate}
\end{definition}
Up to constant scaling of the parameter function $\mu$, all the
algebras in table \ref{ta:3.3} appear in this way. {The
  explicit classification of cuspidal triples when $G$ is simple,
  along with the corresponding values for the parameters $\mu_\al$ can
  be found in the tables of \cite[2.13]{L2}.}

\subsection{}
Consider  the algebraic variety
\begin{align}
&\dot{\fg}=\{(x,gP)\in \fg\times G/P: ~\text{Ad}(g^{-1})x\in \C
  C+\fk t+\fk u\},
\end{align} 
on which $G\times \bC^\times$ acts via $(g_1,\lambda)$: $x\mapsto
\lambda^{-2}\text{Ad}(g_1)x,$ $x\in \fg,$ and $gP\mapsto g_1gP,$ $g\in
G.$

If $\C V$ is any $G\times \bC^\times-$stable subvariety of $\fg,$ we
denote by $\dot{\C V}$ the preimage under the first
projection $\pr_1:\dot\fg\to\fg$. Let $\fg_N$ denote the variety of nilpotent elements in
$\fg.$ We will use the notation $$\CP_e:=\dot{\{e\}}=\{gP:~ Ad(g^{-1})e\in \C C+\fk u\},$$
for any $e\in\fg_N.$ (The identification is via the second projection
$\dot\fg\to G/P.$) Define also $\C P_e^s=\{gP\in \C P_e: Ad(g^{-1})s\in \fk p\}.$

We consider the projections $\pr_\C C:\wti\fg\to \C C$, $\pr_\C
C(x,g)=\pr_\C C(\text{Ad}(g^{-1})x)$, and $\pr_P:\wti\fg\to\dot\fg,
$ $\pr_P(x,g)=(x,gP)$, where $\wti\fg=\{(x,g)\in \fg\times G:  ~\text{Ad}(g^{-1})x\in \C
  C+\fh+\fk u\}$. They are both
$G\times\bC^\times-$equivariant. Let $\dot{\C L}$ be the
$G\times\bC^\times-$equivariant local system on $\dot\fg$ defined by
the condition $\pr_\C C^*(\C L)=\pr_P^*(\dot{\C L}),$ and let $\dot{\C
  L}^*$ be its dual local system. 

\medskip

The classification of simple modules
for $\bH=\bH(L,\C C,\C L)$ is in \cite{L2,L3,L4}. 
 Let us fix a semisimple element $s\in \fg$ and $r_0\in
\bC^\times$, and let $\C T=\C T_{s,r_0}$ be the smallest torus in
$G\times\bC^\times$ whose Lie algebra contains $(s,r_0).$ Let
$\fg_{2r_0}$ be the set of $\C T-$fixed vectors in $\fg,$ namely
\begin{equation}
\fg_{2r_0}=\{x\in\fg:~ [s,x]=2r_0x\}.
\end{equation} 
Let $C_G(s)\times \bC^\times$ be the centralizer of $(s,r_0)$ in
$G\times \bC^\times.$ Since $s\in \fg,$ $C_G(s)$ is a Levi subgroup of
$G$, hence this centralizer is connected.

The construction of standard modules is in equivariant homology
(\cite[section 1]{L2}). For  $(e,\psi),$ where $e\in\fg_{2r_0}$, and
$\psi\in \widehat A(s,e),$ the standard geometric module is (see
\cite[10.7, 10.12]{L3})
:

\begin{align}\label{eq:3.4.4}
X_{(s,e,\psi)}&=\operatorname{Hom}_{A(s,e)}[\psi:H^{\{1\}}_\bullet(\C P_e^s,\dot{\C
  L})]\\\notag
&=\operatorname{Hom}_{A(s,e)}[\psi: \bC_{(s,r_0)}\otimes_{H^\bullet_{C_G(e)^0}}H_\bullet^{C_{G\times\bC^*}(e)}(\C  P_e,\dot{\C
    L})].
\end{align}

One considers the action of $A_G(e)=C_G(e)/C_G(e)^0$ on the homology
$H^{\{1\}}_\bullet(\C P_e,\dot{\C L}),$ and let $\widehat A(e)_\C L$
denote the representations of $A(e)$ which appear in this way. Note
that the natural map $A(s,e)\to A(e)$ is in fact an injection. Let $\widehat A(s,e)_\C L$
denote the representations of $A(s,e)$ which appear as restrictions of
$\widehat A(e)_\C L$ to $A(s,e).$  By
\cite[8.17]{L3} and \cite[8.10]{L2}, {$X_{(s,e,\psi)}\neq 0$} if and
only if  $\psi\in \widehat A(s,e)_\C L.$ 
One can phrase the
classification as follows.

\begin{theorem}[{\cite[8.10, 8.14, 8.17]{L3}, \cite[1.15]{L4}}]\label{c:3.3} 
A standard module $X_{(s,e,\psi)}$ is nonzero if and
only if  $\psi\in \widehat A(s,e)_\C L.$ When $X_{(s,e,\psi)}\neq 0,$
it has a unique irreducible quotient $L_{(s,e,\psi)}.$ 
This induces a natural one-to-one  correspondence
\begin{equation}
  \begin{aligned}
&\text{Irr}_{r_0}\bH(L,\C C,\C
L)\leftrightarrow\{(s,e,\psi): [s,e]=r_0e,\ s\in \fg \text{  semisimple},~e\in \fg_N,~ \psi\in
\widehat A(s,e)_{\C L}\}/G.    
  \end{aligned}
\end{equation}
\end{theorem}

\subsection{}\label{sec:3.5} We are interested in the $W-$structure of
standard modules.  Let $e\in\fg_N$ be 
given. By \cite{L5} (also \cite[section 24]{L7}), the homology group
$H^{\{1\}}_\bullet(\C P_e,\dot{\C L})$ carries a 
$A(e)\times W$ action. Moreover,
\begin{equation}\label{eq:3.5.3}
\mu(e,\phi):=\operatorname{Hom}_{A(e)}[\phi, H^{\{1\}}_0(\C P_e,\dot{\C L})]
\end{equation}
is an irreducible $W-$representation. The
correspondence $\widehat A(e)_\C L\to \widehat W,$ 
$(e,\phi)\to \mu(e,\phi)$ is the generalized Springer correspondence
of \cite{L5}, and it is a bijection. We summarize in the following statement the relevant results
from \cite{L5,L7} that we need.

\begin{theorem}[{\cite[6.5]{L5}, \cite[24.4]{L7}}]\label{t:3.5}Let $(L,\C
  C,\C L)$ be a   cuspidal triple as before.
\begin{enumerate}
\item If $e\in\fg_N,$ $\phi\in \widehat{A}(e)_\C L,$ then
  $\mu(e,\phi)$ from (\ref{eq:3.5.3}) is an irreducible $W-$representation.
\item The $W-$representation $\mu(e,\phi)$ appears with multiplicity
  one in $H^{\{1\}}_\bullet(\C P_e,\dot{\C L})$.
\item If the $W-$representation $\mu(e',\phi')$ occurs in
\newline $H^{\{1\}}_\bullet(\C
  P_e,\dot{\C L})$, then $e'\in \overline{G\cdot e},$ and if $G\cdot e'=G\cdot
  e,$ then necessarily $\phi'=\phi.$

\item If $\phi'\notin \widehat A(e)_\C L$, then
  $\operatorname{Hom}_{A(e)}[\phi', H^{\{1\}}_\bullet(\C   P_e,\dot{\C L})]=0.$
\end{enumerate}
\end{theorem}
Claim (4) in this theorem is proved in \cite{L7}, as a consequence
of a (generalized) Green polynomials algorithm. 

\smallskip

To transfer these results to the $W-$structure of standard modules, via
(\ref{eq:3.4.4}), we also need the deformation argument of \cite{L3}, 10.13.

\begin{lemma}[\cite{L3},10.13]\label{l:3.5} In the notation of (\ref{eq:3.4.4}),
  there is an isomorphism of $W-$representations
\begin{equation}
  \begin{aligned}
X_{(s,e,\psi)}=&\operatorname{Hom}_{A(s,e)}[\psi:H^{\{1\}}_\bullet(\C
P_e^s,\dot{\C L})]=\operatorname{Hom}_{A(s,e)}[\psi:H^{\{1\}}_\bullet(\C
P_e,\dot{\C L})].    
  \end{aligned}
\end{equation}
\end{lemma}

\subsection{}Now we can recall the classification of tempered
modules for the geometric Hecke algebras as in \cite{L4}. Fix
{$r_0> 0.$} 

\begin{definition}A semisimple element $\sigma\in\fg$ is called hyperbolic
(resp. elliptic) if $\text{ad}(\sigma):\fg\to\fg$ has only real
(resp. imaginary) eigenvalues.
\end{definition}

\begin{theorem}[{\cite[1.21]{L4}}]\label{t:3.6} The module
  $L_{(s,e,\psi)}$, where $\psi\in
  \widehat A(s,e)_\C L,$ is a tempered module for $\bH(L,\C C,\C L)$ if and
  only if there exists a Lie triple $\{e,h,f\}$ in $\fg$, such that
  $[s,h]=0,$ $[s,f]=-2r_0f$, $[s,e]=2r_0e,$ and $s-r_0h$ is elliptic.
In this case,  $L_{(s,e,\psi)}=X_{(s,e,\psi)}.$
\end{theorem}

By $sl(2)-$theory, any middle element $h$ of a Lie triple is
hyperbolic. {The condition that $s-r_0h$ be elliptic,
    implies that if $s$ is  hyperbolic, then in fact $s=r_0h.$} In 
this case  $A(e)=A(s,e).$ Recall also that there is a
one-to-one correspondence between nilpotent $G-$orbits in $\fg$ and
$G-$conjugacy classes of Lie triples. Finally, we may assume that $s$
is in $\fk t,$ and therefore the central character of the irreducible
$\bH(L,\C C,\C L)-$module $L_{(s,e,\psi)}$ is the ($W-$conjugacy
class of the) projection on $\fk t$
of $s.$ (The notation is as in section \ref{sec:3.2}.) 
Putting these together, we have the following corollary.
{
\begin{corollary}\label{c:3.6} The map
\begin{equation}
\{(e,\phi): e\in\fg_N,\phi\in \widehat A(e)_\C L\}/G\longleftrightarrow X_{(r_0h,e,\phi)},
\end{equation}
 is a one-to-one correspondence onto the set of isomorphism classes of 
  tempered irreducible $\bH(L,\C C,\C L)-$modules with real central
  character, on which $r$ acts by $r_0$.
\end{corollary}
}

We can now formulate the main result we will need in the rest of the
paper. This is the generalization of the result in \cite{BM1} for
Hecke algebras with equal parameters, and it also appeared in an
equivalent form in \cite{Ci}.

\begin{proposition}\label{p:3.6} Assume $r_0\neq 0.$ The lowest
  $W-$type correspondence 
\begin{equation}
X_{(r_0h,e,\phi)}\longrightarrow \mu(e,\phi)
\end{equation}
between tempered modules in $\text{Irr}_{r_0}\bH(L,\C C,\C L)$ with
real central character and $\widehat W$ is a bijection. Moreover, the
map
 \begin{equation}
X_{(r_0h,e,\phi)}\longrightarrow X_{(r_0h,e,\phi)}|_W
\end{equation}
is uni-triangular with respect to the closure ordering of nilpotent
orbits and the lowest $W-$type map. In particular, the set of tempered
modules with real central character in $\text{Irr}_{r_0}~\bH(L,\C C,\C
L)$ are linearly independent in the Grothendieck group of $W.$
\end{proposition}

\begin{proof}
Using Lemma \ref{l:3.5} and Theorem \ref{t:3.5}, we see that every
tempered module $X_{(r_0h,e,\phi)}$ as in Corollary \ref{c:3.6} has a
unique lowest $W-$type (with respect to the closure ordering of
nilpotent orbits) $\mu(e,\phi)$, and this appears with multiplicity one. The
unitriangularity is also clear from the same results.
\end{proof}

\subsection{}\label{sec:3.7} 
{
We would like to extend Proposition
\ref{p:3.6} to graded affine Hecke algebras $\bH':=\bH\rtimes\Gamma$
of the type as in section \ref{sec:2.4}, under the assumption that
$\bH$ is of geometric type. 

We wish to study the $W'-$structure of tempered $\bH'-$modules with
real central character. For this, we need some elements of
Clifford-Mackey theory, as in \cite{RR}. We recall the
general setting, for an extended algebra $K':=K\rtimes \Gamma,$ which
will then be specialized first to $K=\bC[W]$ and $K'=\bC[W'],$ and
then to $K=\bH$ and $K=\bH'.$

Let $K$ be a finite dimensional $\bC-$algebra, with an action by
algebra automorphisms by $\Gamma,$ and set $K'=K\rtimes\Gamma$. If $V$
is a finite dimensional module of $K,$ and $\gamma\in \Gamma,$ then
let $^\gamma\!V$ denote the $K-$module with the action $x\circ
v:=\gamma^{-1}(x)v,$ $x\in K,$ $v\in V.$ Clearly, $V$ is irreducible
if and only if $^\gamma\!V$ is irreducible. 
Assume $V$ is an irreducible $K-$module. Define the inertia subgroup
\begin{equation}
\Gamma_V=\{\gamma\in\Gamma:~ V\cong ^\gamma\!V\}.
\end{equation}
Since $V$ is simple, the isomorphism $V\to ^\gamma\!V$, for $\gamma \in
\Gamma_V$, is unique up to a scalar multiple. We fix a family of
isomorphisms $\{\tau_\gamma:V\to
^{\gamma^{-1}}\!V\}_{\gamma\in\Gamma_V}$, and define the factor set 
\begin{equation}
\beta:\Gamma_V\times\Gamma_V\to\bC^*,\ \text{ such that
}\tau_\gamma\tau_{\gamma'}=\beta(\gamma,\gamma') \tau_{\gamma\gamma'}. 
\end{equation}
Let $(\bC\Gamma_V)_{\beta^{-1}}$ be the algebra with basis
$\{\gamma:\gamma\in\Gamma_V\}$ and multiplication
\begin{equation}
\gamma\cdot\gamma'=\beta(\gamma,\gamma')^{-1}(\gamma\gamma'),
\end{equation}
where the latter multiplication is understood in $\bC\Gamma_V.$ 
Up to algebra isomorphism, the algebra
$(\bC\Gamma_V)_{\beta^{-1}}$ is independent of the choice of the
family $\{\tau_\gamma\}.$ 
If $U$ is any irreducible $(\bC\Gamma_V)_{\beta^{-1}}-$module, then
there is a natural $K\rtimes\Gamma_V$ action on $V\otimes U$: $x\gamma
(v\otimes u):=(x\tau_\gamma v)\otimes (\gamma u),$ for $x\in K,$
$\gamma\in\Gamma_V.$

\begin{theorem}[{\cite[A.6]{RR}}]\label{t:3.7} With the same notation as above, define
  the induced module
$$V\rtimes
  U:=\operatorname{Ind}_{K\rtimes\Gamma_V}^{K\rtimes\Gamma}(V\otimes U).$$
Then

\begin{enumerate}
\item[(a)] $V\rtimes U$ is an irreducible $K\rtimes\Gamma-$module.
\item[(b)] Every irreducible $K\rtimes\Gamma-$module appears in this
  way.
\item[(c)] If $V\rtimes U\cong V'\rtimes U',$ then $V,V'$ are
  $\Gamma-$conjugate, and $U\cong U'$ as $(\bC\Gamma_V)_{\beta^{-1}}-$modules.
\end{enumerate}
\end{theorem}

We go back to our setting. Set $K=\bC[W]$ first. For every $\mu\in
\widehat W,$ let $\Gamma_\mu$ be the inertia group, and fix a family
of isomorphisms $\{a^\mu_\gamma:\mu\to~
^{\gamma^{-1}}\!\mu\}_{\gamma\in\Gamma_\mu}.$ Since the action of
$\Gamma$ on $W$ comes from the action of $\Gamma$ on the root datum,
we have the following lemma.

\begin{lemma}\label{l:3.7} In the notation of Theorem \ref{t:3.5}, if
  $\mu=\mu(e,\phi)\in\widehat W,$ for some $e\in\fg_N$ and $\phi\in \widehat
  A(e)_\C L,$ then for every $\gamma\in\Gamma$ we have
  $^\gamma\!\mu=\mu(\gamma e,^\gamma\!\phi).$ 
\end{lemma}

When $\bH$ has equal parameters, this statement is a particular case
of \cite[Propositions 2.6.1 and 2.7.3]{Re}. In more generality, one
can follow the analogous argument using the construction of the
$W$-action from \cite[3.4, 8.1]{L2}, and the formal functorial
properties of equivariant homology from \cite[section 1]{L2}.   
\medskip

Now, let us specialize to $K=\bH,$ so that $K'=\bH'.$ Assume $V$ is a
tempered module of $\bH$ with real central character. Then also
$^\gamma\! V$ is tempered with real central character for any
$\gamma\in \Gamma.$ By Proposition
\ref{p:3.6}, there exists a unique lowest $W-$type of $V$, call it
$\mu=\mu(e,\phi)$, which appears with multiplicity one. To emphasize
this correspondence, we write $V=V({e,\phi}).$ 

\begin{proposition}\label{p:3.7}
  We have $\Gamma_{V(e,\phi)}=\Gamma_{\mu(e,\phi)}.$ Moreover, the factor set
  for $\Gamma_{V(e,\phi)}$ can be chosen to be equal to the factor set
  for $\Gamma_{\mu(e,\phi)}.$
\end{proposition} 

\begin{proof} Assume the $W-$structure of $V(e,\phi)$ is 
\begin{equation}
V(e,\phi)|_W=\mu(e,\phi)\oplus\bigoplus_{e\in \overline{G\cdot{e'}}}
m_{(e',\phi')}\mu(e',\phi').
\end{equation}
Let $\gamma\in \Gamma$ be given. Then, by Lemma \ref{l:3.7},
\begin{align}
^\gamma\!V|_W&=\mu(\gamma e,^\gamma\!\phi)\oplus\bigoplus_{e\in \overline{G\cdot{e'}}}
  m_{(e',\phi')}\mu(\gamma e',^\gamma\phi')\\\notag
&=\mu(\gamma e,^\gamma\!\phi)\oplus\bigoplus_{\gamma e\in
  \overline{G\cdot{\gamma e'}}}
  m_{(e',\phi')}\mu(\gamma e',^\gamma\phi'),
\end{align}
where one uses the obvious fact that ${e\in\overline{G\cdot e'}}$ if and only if
${\gamma e\in\overline{G\cdot \gamma e'}}.$ This means that the lowest
$W-$type of $^\gamma\! V$ is $^\gamma\!\mu.$ By Proposition
\ref{p:3.6}, it follows that $V\cong ^\gamma\! V$ if and only if
$\mu\cong ^\gamma\!\mu.$ This proves the first claim in the lemma.

For the second claim, let $\beta$ be the factor set for $\mu$
corresponding to the isomorphisms
$\{a_\gamma^\mu\}_{\gamma\in\Gamma_\mu}.$ Let $\{\tau_\gamma:V\to
^{\gamma^{-1}}\!V \}_{\gamma\in\Gamma_V}$ be a family of isomorphisms
for $V$. Then by restriction to Hom-spaces, we get
\begin{equation}
\Hom_W[\mu:V]\overset{\tau_\gamma}\longrightarrow\Hom_W[\mu:^{\gamma^{-1}}\!
V]\overset{a_\gamma}\longrightarrow\Hom_W[^{\gamma^{-1}}\!\mu:^{\gamma^{-1}}V].
\end{equation}
By Theorem \ref{t:3.5}(2), these spaces are one-dimensional, and so
the composition is a scalar. We normalize $\tau_\gamma$ so that this
scalar equals to one. This forces $\{\tau_\gamma\}$ to have the same
factor set $\beta$ as $\{a_\gamma\}.$ 
\end{proof}

\begin{corollary}\label{c:3.7}\footnote{Recently, 
    \cite{So} gave a proof, by homological methods, of the linear
    $W'$-independence of 
    tempered modules with real central character, but not the ``lowest
  $W'$-type'' uni-triangularity, for graded affine Hecke algebras with
  arbitrary parameters.}
There is a one-to-one correspondence $V'=V{(e,\phi)}\rtimes U\to
\mu'_{V'}=\mu{(e,\phi)}\rtimes U$, $U\in
\widehat\Gamma_{\mu(e,\phi)}$,  between tempered modules with 
real central character for $\bH'=\bH\rtimes\Gamma$ and representations
of $W'=W\rtimes\Gamma.$ Moreover, in the Grothendieck group of $W'$,
the set of tempered $\bH'-$modules with real central character is
linearly independent.
\end{corollary}

\begin{proof}
The first claim follows immediately from Proposition
\ref{p:3.7}, and also the fact that $\mu'_{V'}$ appears with
multiplicity one in $V'.$

The second claim follows immediately, once we define a partial ordering on
$\widehat {W'}$ by setting $\mu'_{(e_1,\phi_1)}<\mu'_{(e_2,\phi_2)}$ if and
only if  ${e_1}\in \overline{G\cdot{e_2}}.$ Then $\mu'_{V'}$ is the lowest
$W'-$type of $V'$ with respect to this order, and the restriction map $V'\to
V'|_{W'}$ is uni-triangular. 

\end{proof}

\section{Signature characters}\label{sec:4}

 In the previous sections, we
encountered three types of graded Hecke algebras associated to a
root system, which appear naturally in the reductions of section
\ref{sec:2.3} from the affine Hecke algebra attached to a root
datum. There is the usual graded Hecke algebra $\bH$, Definition 
(\ref{eq:3.3.1}), the extended Hecke algebra $\bH'=\bH\rtimes \Gamma,$
by a group $\Gamma$ of automorphisms of the root system for $\bH,$ and
also, an induced graded Hecke algebra, which we denote now $\wti \bH'=\C
M_n\otimes_\bC \bH'$, where $\C M_n=\{E_{i,j}\}$ is a matrix
algebra. The finite group parts of these algebras are denoted by
$\bC[\bW]$, $\bC[\bW']=\bC[\bW\rtimes\Gamma],$ and $\bC[\wti \bW'],$ for
$\bH,$ $\bH'$, and $\wti\bH'$, 
respectively.  Here $\bW'$ is a subgroup of index $n$ in $\wti\bW'.$ 
In particular,
the results of this section apply to section \ref{sec:2} by
specializing $\bW$ to $W(\Psi_\sigma)$, the Weyl group of the root
system in section \ref{sec:2.4}, $\Gamma$ to $\Gamma_\sigma$ from
(\ref{2.4.4}), so that $\bW'$ is specialized to $C_{W'}(\sigma)$ from
(\ref{eq:2.3.16}), and $\wti \bW'$ to $W'$ from the paragraph before (\ref{eq:2.2.1a}).

\subsection{}\label{sec:4.1} The Langlands classification for
$\bH$ is in \cite[Theorem 2.1]{Ev}. We need to formulate it in the setting of
$\bH'=\bH\rtimes\Gamma$. 

Recall that the graded Hecke algebra $\bH$ corresponds to a root
system $(\fk t^*,R,\fk t,R^\vee),$ simple roots $\Pi$, and parameter set $\mu.$ Assume the
indeterminate $r$ acts by $r_0\neq 0.$ Let $\Pi_P\subset \Pi$
be given, then we define a parabolic subalgebra $\bH_P\subset \bH,$
which, as a $\bC-$vector space is $\bH_P=\bC[W_P]\otimes S(\fk t^*),$
where $W_P\subset W$ is the subgroup generated by the reflection in
the roots of $\Pi_P.$ The parameter set on $\bH_P$ is obtained
by restriction from $\mu.$  Define
\begin{equation}
\begin{aligned}
&\fk a_P=\{x\in\fk t:~ \langle\al,x\rangle=0,\forall\al\in\Pi_P\},
&\fk a_P^* =\{\omega\in\fk t^*:~
\langle\omega,\check\al\rangle=0,\forall\al\in\Pi_P\},\\
&\fk t_{M}^*=\{\omega\in\fk t^*:~ \langle\omega,x\rangle=0,\forall x\in\fk a_P^*\},
&\fk t_{M}=\{x\in\fk t:~ \langle\omega,x\rangle=0,\forall\omega\in\fk a_P^*\}.
\end{aligned}
\end{equation} 
Then $\bH_P=\bH_M\otimes S(\fk a_P^*)$ as $\mathbb C$-algebras, where $\bH_M=\bC[W_P]\otimes
S(\fk t_M^*)$ as a vector space (and defining commutation relations
coming from $\bH$). 

{The Langlands classification for $\bH$ takes the following form.

\begin{theorem}[{\cite[Theorem 2.1]{Ev}}]\label{Lan}
\begin{enumerate} 
\item[(i)] Let $V$ be an
    irreducible $\bH-$module. Then $V$ is a quotient of a standard
    induced module
    $I(P,U,\nu):=\bH\otimes_{\bH_P}(U\otimes \bC_\nu),$ where $U$
    is a tempered $\bH_M-$module, and $\nu\in\fk a_P^+$, where
\begin{equation}\label{La+}
\fk a_P^+=\{x\in \fk a_P:~ \langle\al,Re ~\nu\rangle>0, \text{ for all }\al\in 
    \Pi\setminus\Pi_P.\}
\end{equation}
\item[(ii)] Any standard module $I(P,U,\nu)$ as in (i) has a unique
  irreducible quotient, denoted $J(P,U,\nu).$
\item[(iii)] $J(P_1,U_1,\nu_1)\cong J(P_2,U_2,\nu_2)$ if and only if
  $\Pi_{P_1}=\Pi_{P_2},$ $U_1\cong U_2$ as
  $\bH_M-$modules, and $\nu_1=\nu_2.$  
\end{enumerate}
\end{theorem}
Moreover, the quotient $J(P,U,\nu)$ is characterized by the following
property. Define an order relation $>$ on real parameters $\fk t_{\bR}$ by 
\begin{equation}
\nu>\nu_0\text{ if } \langle \nu-\nu_0,\alpha\rangle>0,\text{ for
  all }\alpha\in\Pi.
\end{equation}
Then if $J(P_0,U_0,\nu_0)$ is an irreducible subquotient of
$I(P,U,\nu)$ different than $J(P,U,\nu)$, we have $Re\nu>Re\nu_0$.

The case of $\bH'$ compared to $\bH$ is analogous to the case of a
nonconnected $p$-adic group compared to its identity component. In the
setting of $p$-adic groups, the Langlands classification for
nonconnected groups with abelian group of components was carried out
in \cite{BJ1}, starting with the known case of connected groups. We
follow the same approach and simply translate into the graded Hecke
algebra language the results from $p$-adic groups, the proofs being
completely analogous. 

To account for the action of $\Gamma$, define first
\begin{equation}
\Gamma_P=\{\gamma\in\Gamma:~ \gamma\cdot \Pi_P= \Pi_P\}.
\end{equation}
The problem is of course that while any two elements of $\fk a_P^+$
are not conjugate under $W$, they may be conjugate under
$\Gamma_P$. In order to address this, one can define an order $\prec$ on
$\Pi$ and extend this as a lexicographic order with respect to
$\langle~,~\rangle$ on $\fk a_P$. Then we can consider the convex
subchamber of $\fk a_P^+$: 
\begin{equation}\label{La+gamma}
\fk a_P^+(\Gamma)=\{x\in \fk a_P^+: x\preceq \gamma\cdot x,\text{ for all
}\gamma\in \Gamma_P\}.
\end{equation}
One can also define a variant of the order $>$ from before (\cite[Lemma
3.3]{BJ2}). If $\nu,\nu_0$ are real parameters, then write $\nu>_C\nu_0$ if
there exist $\gamma,\gamma_0\in \Gamma$ such that
$\gamma\cdot\nu>\gamma_0\cdot\nu_0$, and one shows that this is well-defined.

\begin{definition}\label{datum}We define a Langlands datum $(P,U',\nu)$ for $\bH'$ to be a triple $\nu,\bH_{P,\nu}',U',\nu$
where:
\begin{enumerate}
\item[(a)] $\Pi_P\subset \Pi$;
\item[(b)] $\nu\in \fk a_P^+(\Gamma)$;
\item[(c)] $\bH_{P,\nu}'=\bH_P\rtimes \Gamma_{P,\nu}$, where
  $\Gamma_{P,\nu}=\{\gamma\in\Gamma_P: \gamma\cdot\nu=\nu\}$;
\item[(d)] $U'$ is irreducible tempered for $\bH_{P,\nu}'$.
\end{enumerate}
\end{definition}

Then the reformulation of the Langlands classification for
non-connected $p$-adic groups in this setting is the following.

\begin{theorem}[{\cite[Theorem 4.2]{BJ1}, \cite[Theorem 3.4]{BJ2}}]\label{extLan}
  Assume that $\Gamma$ is abelian.
\begin{enumerate} 
\item[(i)] Let $V'$ be an
    irreducible $\bH'-$module. Then $V'$ is a quotient of a standard
    induced module
    $I(P,U',\nu):=\bH'\otimes_{\bH'_{P,\nu}}(U'\otimes \bC_\nu),$ for
    a Langlands datum $(P,U',\nu)$ as in Definition \ref{datum}.
\item[(ii)] Any standard module $I(P,U',\nu)$ as in (i) has a unique
  irreducible quotient, denoted $J(P,U',\nu).$ This appears with
  multiplicity one in $I(P,U',\nu)$. Moreover, if $(P_1,U'_1,\nu_1)$
  is the Langlands datum for a different irreducible subquotient than
  $J(P,U',\nu),$ then $Re\nu>_C Re\nu_1.$
\item[(iii)] $J(P_1,U'_1,\nu_1)\cong J(P_2,U'_2,\nu_2)$ if and only if
  $\Pi_{P_1}=\Pi_{P_2}$, $\nu_1=\nu_2,$ and  $U_1'\cong U_2'$ as
  $\bH_{P_1,\nu_1}'-$modules.
\end{enumerate}
\end{theorem}

\begin{remark}\label{r:4.1.4}
\begin{enumerate}
\item In the known applications to $p$-adic groups, the group $\Gamma$
  arises as the group of components of the centralizer of a semisimple
  element in a complex connected reductive group $\frak G$. By a
  classical result of Steinberg (\cite{St}), $\Gamma$ can be embedded
  naturally into the fundamental group of $\frak G$, and therefore it
  is abelian. 

\item If $\bH$ is assumed simple, the only case not covered by Theorem
  \ref{extLan} is when $\bH$ is of type $D_4$ (with equal parameters $1$)
  and $\Gamma=S_3$. But in this case, the algebra
  $\bH'=\bH(D_4)\rtimes S_3$ is known to be isomorphic with the graded
  Hecke algebra of type $F_4$ with parameters $1$ on the long roots
  and $0$ on the short roots. Therefore, with this identification, one may use Theorem \ref{Lan}
  directly in that case. There is another instance of this phenomenon:
  if $\bH'=\bH(A_1)^n\rtimes S_n$, where $S_n$ acts by permuting the
  $A_1$-factors, then $\bH'$ is isomorphic with the graded Hecke
  algebra of type $C_n$ with parameters $0$ on the short roots, and
  $1$ on the long roots.
\item In a recent preprint of Solleveld, \cite[Lemma 2.2.6(b)]{So2}, a slightly different form of the
  maximality of the Langlands parameter  from Theorem \ref{extLan}(ii)
  is proved for arbitrary $\Gamma$, namely that $||\nu||>
  ||\nu_1||$.
\end{enumerate}
\end{remark}

}

One can transfer the classification to $\wti\bH'$ as well, via the
functor
\begin{equation}\label{eq:4.1.3}
\mathcal T :(\bH'-\text{mod})\to (\wti\bH'-\text{mod}),\quad V'\mapsto \wti
V':=\bC^n\otimes_\bC V',
\end{equation}
where $\bC^n=\text{span}_\bC\{E_{i,1}: i=1,n\}$ is, up to isomorphism, the unique irreducible module of $\C M_n.$

\subsection{}\label{sec:4.2} Let $\C R$ denote any one of the three
algebras and let $\C W$ denote its finite  part. One of the arguments
below relies on an induction on the length of the parameter in
Langlands classification. When $\Gamma$ is abelian or if $\bH$ is simple, we can use 
Theorem \ref{extLan}(ii) and Remark \ref{r:4.1.4}(2). In more generality,
one can use the maximality of the Langlands parameter in the form from
Remark \ref{r:4.1.4}(3). For every
irreducible $\C R-$module, we define the notions of $\C W-$character
and signature. The idea is due to \cite{V1}, and it was used in the
$p-$adic and Hecke algebra cases by \cite{BM1}.

\begin{definition}Let $(\pi,V)$ be a finite dimensional
   $\C R-$module. For every irreducible representation
  $(\delta,V_\delta)$ of $\C W,$ set $V(\delta)=\Hom_{\C
    W}[V_\delta,V],$ and let $m(\delta)=\dim_\bC V(\delta)$ be
  the multiplicity of $\delta$ in $V$. We define the $\C
  W-$character of $(\pi,V)$ to be the formal combination
\begin{equation}
\theta_\C W(V)=\sum_{\delta\in\widehat{\C W}}m(\delta)\delta.
\end{equation}
\end{definition}

Recall from section \ref{sec:2} that the algebra $\C R$ has a
$*-$operation (defined explicitly), and that we defined hermitian and
unitary $\C R-$modules with respect to it.

\begin{definition} Let $(\pi,V)$ be a finite dimensional
   $\C R-$module with a nondegenerate hermitian form
  $\langle~,~\rangle.$ For every $(\delta,V_\delta)\in\widehat{\C W}$, fix a
  positive definite hermitian form on $V_\delta.$ Then the space
  $V(\delta)$ acquires a nondegenerate hermitian form, and let
  $(p(\delta),q(\delta))$, $p(\delta)+q(\delta)=m(\delta)$, be its signature. Define the signature
  character of $(\pi,V,\langle~,~\rangle)$ to be the formal pair of sums: 
\begin{equation}
\Sigma(V)=(\sum_{\delta\in\widehat{\C W}}p(\delta)\delta,\sum_{\delta\in\widehat{\C W}}q(\delta)\delta).
\end{equation}
\end{definition}

It is clear that $(\pi,V,\langle~,~\rangle)$ is unitary if and only if
$q(\delta)=0$ for all $\delta\in\widehat{\C W}.$ The fundamental
result that one needs relates the signature of an irreducible module
to a combination of signatures of the tempered modules. In the
original real groups setting, this is \cite[Theorem 1.5]{V1}. In the
setting of representations of $p-$adic groups with Iwahori fixed vectors and affine Hecke algebras, this is \cite[Theorem
5.2]{BM1}.  

\begin{theorem}[\cite{V1,BM1}]\label{t:4.2} Let $(\pi,V)$ be an irreducible $\C  R-$module having a nonzero hermitian form. Then there exist finitely
  many 
  irreducible tempered modules $(\pi_j,V_j),$ and integers
  $a_j^\pm,$ $j=1,m$, such that
$$\Sigma(V)=(\sum_{j=1}^m a_j^+\theta_\C W(V_j),\sum_{j=1}^m a_j^-\theta_\C W(V_j)).
$$
Moreover, if $(\pi,V)$ has real central character, then so do $(\pi_j,V_j),$ $j=1,m.$
\end{theorem}

The proof of the theorem is formal, once the results in section
\ref{sec:4.1} are given, see \cite[section 5]{BM1}. It uses the notion of Jantzen filtration, and
an induction on the length of the Langlands parameter.

\subsection{}\label{sec:4.3} We compare the signature characters for
the Hecke algebras $\C R$. Let $\overline\chi$ denote a central
character for $\C R.$

\begin{definition}\label{d:vogan} Let $(\pi,V)$ be an irreducible $\C R-$module
  with a nonzero hermitian form and corresponding signature character
  $\Sigma(V)$ as in Theorem \ref{t:4.2}. We say that $V$ has the Vogan property, if 
$$\sum_{j=1}^m a_j^-\theta_\C W(V_j)=0\text{ implies that }
a_j^-=0,\text{ for all }j=1,m.
$$
We say that a subcategory $\text{mod}_{\overline\chi}\C R$ has the
  Vogan property if every irreducible module in this subcategory does.
\end{definition}

\begin{lemma}\label{l:4.3}
Let $\C R$ be $\bH$ or $\bH'.$ Then the category of finite dimensional
$\C R-$modules with
real central character has the Vogan property.
\end{lemma}

\begin{proof}
This is immediate by Proposition \ref{p:3.6} and Corollary
\ref{c:3.7}, which say that the set of tempered $\C R-$modules with
real central character have linearly independent $\C W-$characters.
\end{proof}

We cannot apply the same argument for $\wti \bH'.$ The reason is the
following. Note that, by Corollary \ref{c:2.6}, in the correspondence
(\ref{eq:4.1.3}), we have
\begin{equation}
\theta_{\wti \bW'}(\wti V')=\operatorname{Ind}_{\bW'}^{\wti \bW'}(\theta_{\bW'}(V')).
\end{equation}
Consequently, it is possible to have distinct tempered modules with real central character for
$\wti \bH'$, which have the same $\wti \bW'$-structure. In particular,
the linear $\wti \bW'$-independence for tempered modules does not hold
anymore. 
\begin{example}
In the notation of section \ref{sec:2}, let $\C H(\Psi,\sqrt q)$ be the affine Hecke
algebra for the root datum $\Psi$ of $G=Sp(2m,\mathbb C)$ and
constant parameter functions $\lambda=\lambda^*\equiv 1.$ Set
$\Gamma=\{1\}$. Let $\sigma$ be an elliptic semisimple element of $G$
such that $C_G(\sigma)=Sp(2k,\mathbb C)\times Sp(2m-2k, \mathbb
C)$. We have $\Gamma_\sigma=\{1\}$, $\Psi_\sigma$ is the root datum of
$C_G(\sigma),$ and $\mu_\sigma$ is the constant function $2$. In the
notation of this section: $\wti\bW'=W(C_m)$ and $\bW'=W(C_k)\times
W(C_{m-k})$, $\mathbb H'=\mathbb H_{\mu_\sigma}(\Psi_\sigma)=\mathbb
H(C_{k})\times \mathbb H(C_{m-k})$, and $\wti\bH'\cong \C
M_n\otimes_\bC \bH',$ where $n={m\choose k}$. By Proposition \ref{p:3.6}, there
are $|\widehat W(C_k)| |\widehat W(C_{m-k})|$ distinct tempered
modules with real central character for $\mathbb H'$, and by Lemma
\ref{l:2.7}, these give rise precisely to the tempered modules with real central character for $\wti\bH'$. Therefore, whenever we have $|\widehat W(C_k)| |\widehat W(C_{n-k})|>|\widehat W(C_n)|,$
the tempered $\wti \bH'$-modules with real central character are not linearly $\wti\bW'$-independent. For example, this is the case if $k=2$ and $m=4.$ 
\end{example}

One bypasses this difficulty as in \cite{BM2}.

\begin{proposition}\label{p:4.3}
The category of finite dimensional $\wti\bH'-$modules with real
central character has the Vogan property.
\end{proposition}

\begin{proof}
{
Let $\wti V'$ be an irreducible hermitian $\wti\bH'-$module with real central character. Assume that
$$\Sigma(\wti V')=(\sum_{j=1}^m a_j^+\theta_{\wti \bW'}(\wti
V'_j),\sum_{j=1}^m a_j^-\theta_{\wti \bW'}(\wti V'_j)).$$ Applying the
comparison theorems for parabolic induction and tempered spectra  from
\cite[section 7, Theorem 7.2 and Corollary 7.3]{BM2} to the 
functor  $V'\mapsto\wti V'$ in (\ref{eq:4.1.3}), one finds that 
{
$$
\Sigma(V')=(\sum_{j=1}^m a_j^+\theta_{\bW'}(
V'_j),\sum_{j=1}^m a_j^-\theta_{\bW'}(V'_j)).
$$
}
By Lemma \ref{l:2.7}, all
$V_j'$ are tempered with real central character. 
If $\sum_{j=1}^m a_j^-\theta_{\wti \bW'}(\wti V'_j)=0,$ then
$\wti V'$ unitary, and so $V'$ is unitary too. This means $ \sum_{j=1}^m
a_j^-\theta_{\bW'}(V'_j)=0.$ The claim follows now by Lemma \ref{l:4.3}.
}
\end{proof}

\subsection{}\label{sec:appl} 
The main application is next. Call an affine algebra $\C
H^{\lambda,\lambda^*}(\Psi)$ of geometric type, if all the graded
algebras $\bH_{\mu_\al}$ that appear via the reductions in section
\ref{sec:2} are of geometric type in the sense of section \ref{sec:3}.

\begin{theorem}\label{c:4.3} Assume the notation from section \ref{sec:2}. Let
  $\C H^{\lambda,\lambda^*}(\Psi,z_0)$ be the affine Hecke algebra attached to a root datum
  $\Psi$ (Definition \ref{d:2.1}) and with parameters $\lambda,\lambda^*$  of geometric
  type (section \ref{sec:3}), and assume $z_0$ is not a root of unity,
  and let $\C H'(\Psi,z_0)=\C H^{\lambda,\lambda^*}(\Psi,z_0)\rtimes \Gamma$ be an extended Hecke algebra. Let $s_e\in T_e$ be a
  fixed elliptic semisimple element. The
  category of finite dimensional $\C
  H^{\lambda,\lambda^*}(\Psi,z_0)-$modules whose central characters have
  elliptic parts $G(\Psi)\rtimes\Gamma$-conjugate to $s_e$ has the Vogan property
  (Definition \ref{d:vogan}).
\end{theorem}

\begin{proof}
This follows now from Propositions \ref{p:2.5}, \ref{p:2.6}, and \ref{p:4.3}.
\end{proof}

\section{Elements of the theory of types}\label{sec:5} We need to recall certain
elements of the theory of types. The background material for section \ref{sec:5} is well-known, e.g., \cite{BK1,BK2,BHK}.

\subsection{} Let $\bF$ denote a $p-$adic field, with norm $||~||$, and let
$\CG$ be the group of $\bF-$rational points of a connected, reductive,
algebraic group defined over $\bF.$ Let $\fk R(\CG)$ denote the
category of smooth complex representations of $\CG.$ 

A character
$\chi:\CG\to \bC^\times$ is called unramified, if there exist 
$\bF-$rational characters $\phi_j:\CG\to \bF^\times$, $j=1,k$, and complex
numbers $s_j$, $j=1,k$, such that
$\chi(g)=\prod_{j=1}^k||\phi_j(g)||^{s_j}$, for all $g\in \CG.$

We recall Bernstein's decomposition of $\fk R(\CG)$ adapted to the
theory of types. One defines an equivalence
relation on the set of pairs $(L,\sigma),$ where $L$ is a $\bF-$rational Levi
subgroup of $\CG$, and $\sigma$ is an irreducible supercuspidal representation
of $L$ as follows. 

\begin{definition}Two pairs $(L_1,\sigma_1)$ and $(L_2,\sigma_2)$ are
inertially equivalent if there exists $g\in \CG$, and an unramified
character $\chi$ of $L_2$ such that
$gL_1g^{-1}=L_2,$ and $g\cdot \sigma_1=\sigma_2\otimes\chi.$ 
\end{definition}

 Let $\fk B(\CG)$ denote the set of inertial equivalence
classes. If $(\pi,V)$ is an irreducible representation in $\fk R(G),$
then there exists a pair $(P,\sigma),$ where $P=LN$ is a
$\bF-$rational parabolic and $\sigma$ is an irreducible supercuspidal
representation of $L,$ such that $\pi$ is equivalent with a
subquotient of the normalized induced representation
$\Ind_P^G(\sigma)$. Moreover, the pair $(L,\sigma)$ is unique up to
conjugacy, so in particular, one can assign to $\pi$, in a
well-defined way, the inertial
class $\fk s=[L,\sigma]\in \fk B(\CG)$. This is called the inertial support
of $\pi.$ More generally, one defines a full subcategory $\fk R^\fk
s(\CG)$ whose objects are $(\pi,V)$ (not
necessarily irreducible) for which all irreducible subquotients have
inertial support $\fk s.$ 

\begin{remarks} 
\begin{enumerate}
\item[(1)] The inertial support gives a decomposition
\begin{equation}
\fk R(\CG)=\prod_{\fk s\in \fk B(\CG)}\fk R^\fk s(\CG).
\end{equation}
\item[(2)] This decomposition is well behaved with respect to
Langlands classification. Since all subquotients of a parabolically
induced module have the same inertial support (\cite[Theorem
2.9]{BZ}), if the Langlands subquotient
of a standard module has inertial support $\fk s,$ then the standard
module is in $\fk R^\fk s(\CG).$ 
\end{enumerate}
\end{remarks}

\subsection{} Let $\CH(\CG)$ denote
the Hecke algebra of $\CG$, \ie the space of locally constant,
compactly supported complex functions on $\CG$ with convolution with
respect to some fixed Haar measure. The algebra $\CH(\CG)$ has a
natural $*$-operation: if $f\in\CH(\CG),$ then $f^*(g):=\overline
{f(g^{-1})}.$ Let $J$ be a compact open subgroup
of $\CG$, and fix a smooth irreducible representation $(\rho,\C W)$ of
$J$. Let $(\rho^\vee, \C W^\vee)$ be the contragredient
representation.

\begin{definition}\label{d:startype} Let $\CH(\CG,\rho)$ be the vector space of compactly
  supported functions $f:\CG\to \End_\bC(\C W^\vee)$ satisfying 
$$f(j_1gj_2)=\rho^\vee(j_1)\circ f(g)\circ \rho^\vee(j_2),\ j_1,j_2\in
J, g\in \CG.$$
Under the convolution, $\CH(\CG,\rho)$ becomes an associative algebra
with unit given by the function $1_\rho(g)=\frac 1{\text{vol}(J)}\rho^\vee(g)$, if $g\in J$, and $1_\rho(g)=0$, if $g\notin J$. Let $^*$ denote the
adjoint involution on $\End_\bC(\C W^\vee)$ with respect to a fixed $J$-invariant
positive definite form on $\C W^\vee.$ Then one can define
$*$-operation on $\CH(\CG,\rho)$ by $f^*(g):=
{f(g^{-1})^*}$, $f\in \CH(\CG,\rho),$ $g\in \CG$, and an inner product
$[~,~]$ given by $$[f,h]:=\frac{\text{vol}(J)}{\dim \rho}
\tr((f^*\star h)(1_\CG)),\quad f,h\in \CH(\CG,\rho);$$
here $1_\CG$ denote the identity in $\CG.$ This makes $\CH(\CG,\rho)$
into a normalized Hilbert algebra with an abstract Plancherel formula, see 
\cite[sections 3 and 4]{BHK}. In particular, one can define tempered, hermitian and unitary
$\CH(G,\rho)$-modules. 

\end{definition}

There is a natural isomorphism
$\CH(\CG,\rho)\cong \End_G(\cInd_J^G(\rho))$, where $\cInd$ denotes
compact induction. 
If $(\pi,V)\in \fk R(G)$, define the space of $\rho$-invariants in $V$:
\begin{equation}
V_\rho=\Hom_J[\C W,V].
\end{equation}
This space is a left $\CH(\CG,\rho)-$module as
follows. If $\phi\in \Hom_J[\C W,V]$ and $f\in \CH(\CG,\rho)$, then
$\pi(f)\phi$ is the homomorphism
\begin{equation}
\C W\ni w\mapsto \int_G\pi(g)~ \phi((f(g))^\vee w)~dg.
\end{equation}

\subsection{}Let $\fk R_\rho(\CG)$ denote the full subcategory of $\fk R(\CG)$
whose objects are representations $(\pi,V)$ such that $V$ is generated
by the $\rho$-isotypic component of $V$. One has the functor: 
\begin{align}\label{d:Mrho}
&\mathbf M_\rho:\fk R_\rho(\CG)\to \CH(\CG,\rho)-\text{mod},\
&\mathbf M_\rho(V)=V_\rho.
\end{align}
 
\begin{definition}\label{d:type}
The pair $(J,\rho)$ is called a type in $\CG$ if the category $\fk R_\rho(\CG)$
is closed under subquotients.
\end{definition}

\begin{theorem}[{\cite[3.12,4.3]{BK2}}]\label{t:BK} Assume that $(J,\rho)$ is a type
  in $\CG$. Then:

\begin{enumerate}

\item [(i)] the functor $\mathbf M_\rho$ from (\ref{d:Mrho}) is an equivalence of categories;

\item [(ii)] there exists a finite subset $\fk S\subset\fk B(\CG)$ such
that 
$\fk R_\rho(\CG)=\prod_{\fk s\in\fk S}\fk R^\fk s(\CG).$ (In
  this case, $(J,\rho)$ is called an $\fk S$-type.)
\end{enumerate}
\end{theorem}

It is known from \cite{BHK} that the functor $\mathbf M_\rho$ induces a homeomorphism of the supports of the Plancherel measures  in the two categories. In particular:

\begin{theorem}[{\cite[Theorem B]{BHK}}]\label{t:bhk} The functor $\mathbf M_\rho$ induces a bijection between the irreducible tempered modules in $\fk R_\rho(\CG)$ and $\CH(\CG,\rho)-\text{mod}$.
\end{theorem}

\subsection{}\label{sec:5.4} 
Let
 $(J,\rho)$ be a $\fk s$-type, where
$\fk s=[L,\sigma]\in \fk B(\CG)$. 
The pair $(L,\sigma)$ also gives rise to an element
$\fk s_L=[L,\sigma]\in \fk B(L)$. Assume  that there
exists a $\fk s_L$-type $(J_L,\rho_L)$
in $L$ such that $(J,\rho)$ is a cover of $(J_L,\rho_L)$, in the sense
of \cite[section 8]{BK2}. Let $P$ be a parabolic subgroup of $\CG$ with Levi
$L$. By \cite[(7.12)]{BK2}, there exists a natural
injective algebra homomorphism
\begin{equation}
t_P:\C H(L,\rho_L)\to \C H(\CG,\rho)
\end{equation}
If the supercuspidal $L$-representation $\sigma$ is obtained by compact
induction as in \cite[(5.5)]{BK2}, then the algebra $\C
H(L,\rho_L)$ is abelian, \cite[Proposition 5.6]{BK2}.

\begin{definition}\label{d:5.3}
The $\fk s$-type $(J,\rho)$ is called {\it affine} if there exists an algebra isomorphism $$\xi: \CH(G,\rho)\to
\CH'(\Psi),$$
where $\CH'(\Psi)$ is an extended affine Hecke algebra from Definition
\ref{d:2.1}\footnote{for some specialized value $z_0\in\bR_{>1}$ of the
indeterminate $z$.} of geometric type  satisfying the properties:
\begin{enumerate}
\item[(i)] $\xi$ is an isomorphism of Hilbert algebras, where the
  Hilbert algebra structure for $\C H(\CG,\rho)$ is as in Definition \ref{d:startype}, while for $\CH'(\Psi)$ it is as in section \ref{sec:2.6}.
\item[(ii)] there exists a compact open subgroup $\C K$ of $\C G$ with
  $J\subset \C K$ and $\CG=\C K P$, such that $\xi(\CH(\C K,\rho))=\CH_{W'},$ where
  $\CH_{W'}$ is as in (\ref{eq:2.2.1}).
\item[(iii)] $\xi(t_P(\CH(L,\rho_L)))=\C A(\Psi),$ where
  $\CA(\Psi)$ is as in section \ref{sec:2.2}.
\end{enumerate}
\end{definition}

The following remark justifies condition (i) in Definition \ref{d:5.3}.
\begin{remark}\label{r:temptype}If $(J,\rho)$ satisfies
    condition (i) in Definition \ref{d:5.3}, then the isomorphism $\xi$
  induces a bijection between the irreducible tempered $\C H(\CG,\rho)$-modules and the irreducible tempered $\C H'(\Psi)$-modules.
\end{remark}

\begin{proof}
The tempered spectrum of $\C H(\CG,\rho)$ is the support of the Plancherel
measure in $\C H(\CG,\rho)$-mod. Furthermore, the tempered spectrum
for $\C H'(\Psi)$ defined in Definition \ref{d:2.7} is also the
support of the Plancherel measure for $\C H'(\Psi)$  (see \cite[Lemma 2.20 and Theorem 2.25]{Op}). The claim now follows from property (i) of $\xi$ in Definition \ref{d:5.3}. 
\end{proof}

\subsection{}\label{s:bm}

\begin{definition}\label{d:unitequiv} Two module categories (as before) are said to be \emph{unitarily equivalent} if their
  unitary, respectively hermitian irreducible modules are in bijection via
  an equivalence of categories.
\end{definition}

We can extend  \cite[Theorem 1.1]{BM1} to obtain the correspondence of
hermitian and unitary irreducible modules between the categories $\fk
R_\rho(\CG)=\fk R^{\fk s}(\CG)$ and $\C H'(\Psi)$-mod$\cong\C
H(\CG,\rho)$-mod. 
With Theorem \ref{c:4.3}, Theorem \ref{t:bhk} and Remark \ref{r:temptype} in hand, the argument of \cite[Theorem 1.1]{BM1} from the
Iwahori case for split adjoint groups can be applied formally to the
setting of an affine type $(J,\rho)$.

\begin{theorem}\label{t:5.3}
If the type $(J,\rho)$ is affine in the
sense of Definition \ref{d:5.3}, then the categories $\fk R^{\fk
  s}(\CG)=\fk R_\rho(\CG)$ and $\C
H(\CG,\rho)$-mod($\cong\C H'(\Psi)$-mod) are unitarily equivalent via
the functor $M_\rho$.
\end{theorem}

\begin{proof}[Sketch of proof] The nontrivial part is to prove that the unitarity of the $\C H(\CG,\rho)$-module implies that unitarity of the representation in $\fk R^{\fk s}(\CG).$ Definition \ref{d:5.3}
  insures that the central characters in the three categories $\fk
  R^{\fk s}(\C G)$, $\C H(\CG,\rho)$-mod and $\CH'(\Psi)$-mod
  correspond. 
Using Theorem \ref{t:bhk}, one can relate the
signature $\C K$-character for an admissible irreducible representation $\pi$ in
$\fk R^{\fk s}(\CG)$ with the signature $\C H(\CK,\rho)$-character of
$\mathbf M_\rho(\pi)$ in $\C H(\CG,\rho)$-mod, see section 5,
particularly pages 32-33, in \cite{BM1}. Via the isomorphism
between $\C H(\CK,\rho)$ and $\C H_{W'}$ in Definition \ref{d:5.3}
(ii), and the correspondence of tempered modules from Remark \ref{r:temptype}, the
signature $\C H(\CK,\rho)$-character of $\mathbf M_\rho(\pi)$ is in
turn identified with the $\C H_{W'}$-signature
character of $\xi(\mathbf M_\rho(\pi))$ in  $\C H'(\Psi)$-mod, \cf \cite[Theorem 5.7]{BM1}. Since the
category $\C H'(\Psi)$-mod (with a fixed elliptic central character) has
Vogan's property by Theorem \ref{c:4.3}, the claim follows. The argument is identical with the one in the proof of Theorem 1.1 of \cite[page 33]{BM1}.
\end{proof}

\section{Unitary correspondences}\label{sec:6}
 
In this section, we give two examples of unitary correspondences as in
Theorem \ref{t:5.3}, and present certain applications to unitary
equivalences with endoscopic groups.

\subsection{Unramified principal series}\label{sec:6.4} 
Assume that $\CG$ is the $\bF$-points of a linear reductive algebraic
group over $\bF.$ Denote by $v_\bF$ the valuation function on $\bF.$
Standard references for the discussion about unramified principal
series are \cite{Ca}, \cite{Bo}.

Fix $A$ a
maximally split torus in $\CG$ and set 
\begin{equation}
M=C_\CG(A), \quad W(\CG,A)=N_\CG(A)/M.
\end{equation} 
Let $X^*(M)$ and $X_*(M)$ denote the lattices of algebraic characters
and cocharacters of $M$, respectively, and $\langle~,~\rangle$ their
natural pairing. Define the valuation map
\begin{equation}
  \begin{aligned}
v_M:&~M\to X_*(M),\langle\lambda,v_M(m)\rangle=v_\bF(\lambda(m)),\text{ for all } m\in M,\ \lambda\in X^*(M).    
  \end{aligned}
\end{equation}
Set $^0\!M=\ker v_M$ and $\Lambda(M)=\text{Im} v_M.$ Similarly, define
$v_A,$ $^0\!A$, and $\Lambda(A).$ Since $A$ is a torus, we have
$\Lambda(A)=X_*(A).$ Moreover, we have $X_*(A)\subset
\Lambda(M)\subset X_*(M).$ (Notice that $\Lambda(M)=X_*(M)$ precisely
when $M=A,$ i.e., $\CG$ is $\bF$-split.)

The group of unramified characters
of $M$ (i.e., characters trivial on $^0\!M$) will be denoted by
$\widehat M^u.$ For every character $\chi\in \widehat M^u$, let
$X(\chi)$ denote the corresponding unramified principal series. It is
clear that 
\begin{equation}
\widehat M^u\cong \Hom(\Lambda(M),\bC^\times),
\end{equation}
so if we define $T'=\text{Spec}\bC[\Lambda(M)]$, a complex algebraic
torus, we have a natural identification
\begin{equation}
\widehat M^u=T'.
\end{equation}

 Let $\CK$ and $\CI$ be a special maximal compact open
subgroup of $\CG$ and an Iwahori subgroup, respectively, attached,
using the Bruhat-Tits building, to the torus 
$A$ and a special vertex $x_0$ (see \cite{Ti}, or \cite[section 3.5]{Ca}). Then $^0\!A=A\cap\CK$ and $^0\!M=M\cap\CK$.  The Weyl group $W(\CG,A)$ acts on $X_*(M)$ preserving $X_*(A)$ and
$\Lambda(M).$ If we let $\wti W(\CG,A)=W(\CG,A)\ltimes \Lambda(M)$ denote the extended
Weyl group, the Bruhat-Tits decomposition is
\begin{equation}
\CG=\CI~\wti W(\CG,A)~\CI \text{ and } \CK=\CI ~W(\CG,A)~ \CI.
\end{equation}

The subquotients of the minimal (unramified)
principal series $X(\chi)$ have inertial
support $\one=[A,\one_A],$ where $\one_A$ denotes the trivial
character on $A.$ In other words, the irreducible subquotients of the minimal
principal series form the irreducible objects of the category $\fk
R^{\one}(\C G).$ 

\begin{theorem}[\cite{Bo,Cas}] The pair $(\CI,1_\CI)$ is a $\one$-type (in the sense of Definition \ref{d:type}) 
  for $\CG$, i.e., $\fk R^\one(\CG)=\fk
  R_{(\CI,1_\CI)}(\CG).$ Moreover, in the equivalence of categories with $\C H(\CG,\C I)$, the tempered modules correspond.
\end{theorem}

In this case, the structure of $\CH(\CG,1_\CI)$ is well-known by \cite{IM}. Its description  with
generators and relations and the explicit parameters are in the tables of \cite{Ti}. In
the terminology of section \ref{sec:2}, it is an affine Hecke algebra
$\CH(\Psi)$ with certain unequal parameters of geometric type for a
root datum $\Psi$. More precisely, with the notation from section
\ref{sec:2}, particularly Definition \ref{d:2.1}, we have
$\Psi=(X,X^\vee,R,R^\vee)$, where
\begin{enumerate}
\item $X=X^*(T')(=\Lambda(M))$, $X^\vee=X_*(T')$;
\item the Weyl group of $\Psi$ is $W(\CG,A)$;
\item $R^\vee$ is the set of ``restricted roots'' of $A$ in $\CG$ (see
  \cite[section 1.9]{Ti} or \cite[page 141]{Ca}).
\end{enumerate}

This implies that
$(\CI,1_\CI)$ is an affine $\one$-type in the sense of Definition
\ref{d:5.3}. 
Thus we have: 

\begin{theorem}\label{t:6.1} The  categories $\fk
  R^\one(\CG)$ and  $\CH(\CG,1_\CI)$ are unitarily equivalent.
\end{theorem}

\begin{remark}\label{r:6.1}Following Theorem \ref{t:2.2}, we see that the central
  characters of $\CH(\CG,1_\CI)$ are in one-to-one correspondence with
  $W(\CG,A)$-conjugacy classes in $T'.$ We will use this fact in
  section \ref{sec:quasi}.

\end{remark}

\subsection{Quasisplit groups}\label{sec:quasi} We retain the notation from \ref{sec:6.4}.  In this
subsection we explain a correspondence between unitarizable principal
series of a quasisplit, nonsplit, quasisimple $p$-adic group $\CG$ and certain endoscopic
split groups. We assume in addition  that $\CG$ splits over an unramified extension of
$\bF.$
The key
observation is that while $\CH(\CG,\one_\CI)$  may have unequal parameters, all of the
graded Hecke algebras attached can be identified naturally with
graded Hecke algebras with equal parameters.
For this we need to examine the Iwahori-Hecke algebra and its graded versions more
closely. Assume that the root datum $\Psi$ for 
$\CH(\CG,1_\CI)$ is non-simply laced root datum, and with
parameters $\lambda$ and $\lambda^*$, the latter occurring when $\Psi$ is of type $B.$ 

 Let $\al_s$, $\al_\ell$ denote a short root and a
long root respectively, and let $|\al|$ denote the squared length of a root $\al.$
The following lemma can be verified by inspecting Tits' tables for
quasisplit groups (\cite{Ti}).

\begin{lemma}\label{l:6.4} The parameters of $\CH(\CG,1_\CI)$ satisfy the
  conditions:
\begin{enumerate}
\item $\frac{\lambda(\al_s)}{\lambda(\al_\ell)}\in
  \left\{1,\frac{|\al_s|}{|\al_\ell|}\right\}$;
\item $\frac{\lambda(\al_s)\pm\lambda^*(\al_s)}{2\lambda(\al_\ell)}\in
  \left\{0,\frac{|\al_s|}{|\al_\ell|}, 1\right\}.$
\end{enumerate}
\end{lemma}

These conditions guarantee that for every graded Hecke algebra $\bH_{\mu_\sigma}$ that
appears in Remark \ref{r:2.3} (via Theorem \ref{t:2.3}), the parameters $\mu_\sigma(\al)$
satisfy one of the following properties:
\begin{description}
\item[(a)] $\mu_\sigma(\al_s)=\mu_\sigma(\al_\ell),$ 
\item[(b)] $\mu_\sigma(\al_s)=\frac{|\al_s|}{|\al_\ell|}
  \mu_\sigma(\al_\ell),$
\item[(c)] $\mu_\sigma(\al_s)=0.$
\end{description}
In case (a), $\bH_{\mu_\sigma}$ is a graded Hecke algebra with equal
parameters. In case (b), we have a natural isomorphism 
\begin{equation}\label{eq:6.4.2}
\bH_{\mu_\sigma}\cong \bH_{\mu_\sigma^\vee}^\vee,
\end{equation}
where $\bH_{\mu_\sigma^\vee}^\vee$ is the graded Hecke algebra attached to
the dual root system to that of $\bH_{\mu_\sigma}$ and with parameters
$\mu_\sigma^\vee(\check\al_\ell)=\mu_\sigma(\al_\ell),$
$\mu_\sigma^\vee(\check\al_s)=\mu_\sigma(\al_s).$ Notice that
$\bH_{\mu_\sigma^\vee}^\vee$ is a graded Hecke algebra with equal
parameters. In case (c), we have a natural isomorphism (see
\cite[Proposition 4.6]{BC})
\begin{equation}\label{eq:6.4.3}
\bH_{\mu_\sigma}\cong \bC[W_s]\ltimes \bH^0_{\mu_\sigma},
\end{equation}
where $W_s$ is the reflection subgroup of $W$ generated by the simple
short roots, and $\bH^0_{\mu_\sigma}$ is the graded Hecke algebra
(with equal parameter $\mu_\sigma(\al_\ell)$) corresponding to the
root system of long roots. 

\medskip

Every unramified principal series $X(\chi)$
contains a unique irreducible ($\CK$-)spherical subquotient $\overline
X(\chi).$  It is
well-known that two unramified principal series $X(\chi)$ and
$X(\chi')$ have the same composition factors, and in particular
$\overline X(\chi)\cong \overline X(\chi')$ if and only if
\begin{equation*}
\chi=w\chi',\text{ for some }w\in W(\CG,A).
\end{equation*}
Let $\widehat\CG_{\text{sph}}$ denote the set of isomorphism classes
of irreducible spherical representations of $\CG.$ Therefore, we have
a  one-to-one correspondence
\begin{equation}\label{eq:6.4.5}
\widehat\CG_{\text{sph}}\longleftrightarrow W(\CG,A)\text{-orbits in }
\widehat M^u=T'.
\end{equation}
(A better way to express this correspondence is via the Satake
isomorphism $\CH(\CG,\CK)\cong
\bC[\widehat{M/^0\!M}]^{W(\CG,A)}$, see for example \cite[Theorem 4.1]{Ca}.)

We need to recast this bijection in terms of the dual L-group. 
Let $G$ denote
the complex connected group dual (in the sense of
Langlands) to $\CG$, and let $T$ be a maximal torus in $G$. Let
$\Psi(G)=(X^*(T),X_*(T),R(G,T),R^\vee(G,T))$ be the corresponding root datum.
The inner class of $\CG$ defines a  homomorphism $\tau:
\Gamma\to \Aut(\Psi(G)),$ where $\Gamma=\Gal(\overline\bF/\bF).$ Since
we assumed that $\CG$ is quasisimple unramified, we know that the image
$\tau(\Gamma)\subset  \Aut(\Psi(G))$ is a cyclic group 
generated by an automorphism of order $d$ ($d\in\{2,3\}$) which, by abuse of notation, we also call
$\tau.$  We fix a choice of 
root vectors $X_\al,$ for $\al\in R.$ The automorphism $\tau$ maps the root space
of $\al$ to the root space of $\tau(\al).$ We normalize $\tau$ such
that 
\begin{equation}
\tau(X_\al)=X_{\tau(\al)}, \text{ for all }\al.
\end{equation}
This allows one to extend $\tau$ to an automorphism of $G$ in a
canonical way.

\begin{definition}
Two elements $x_1,x_2\in G$ are called $\tau$-conjugate if there
exists $g\in G$ such that $x_2=g\cdot x_1\cdot \tau(g^{-1}).$ For a
subset $S\subset G$, denote:
$$N_G(S\tau)=\{g\in G:~ g\cdot S\cdot \tau(g^{-1})\subset S\}.$$
\end{definition}

 The construction of the L-group is such that we have (see
 \cite[section 6]{Bo2}):
\begin{equation}
X^*(T')=X^*(T)^\tau,\quad W(\CG,A)=W(G,T)^\tau.
\end{equation}
In particular, we have an inclusion $X^*(T')\hookrightarrow X^*(T)$,
which gives a surjection $\nu:T\to T'.$ By \cite[Lemma 6.4]{Bo2},
the map 
\begin{equation}
\nu': T\rtimes\langle\tau\rangle\to T',\quad \nu'((t,\tau))=\nu(t),
\end{equation}
induces a bijection of $(N(T\tau),\tau)$-conjugacy classes of elements
in $T$ and $W(G,T)^\tau$-conjugacy classes of elements in $T'.$

\begin{theorem}[Langlands, \cf {\cite[Theorem 3.1]{Ca}}, {\cite[Proposition 6.7]{Bo2}}]\label{t:6.2} There are
  bijective correspondences:
\begin{equation}\label{eq:6.4.6}
\begin{aligned}
&\text{semisimple } \tau\text{-conjugacy classes in }
G&&\longleftrightarrow\\ 
&(N_G(T\tau),\tau)\text{-conjugacy classes of elements in
}T&&\longleftrightarrow \\
&W(G,T)^\tau\text{-orbits in 
}T'(\longleftrightarrow \widehat\CG_{\text{sph}}).
\end{aligned}
\end{equation} 
\end{theorem}

\begin{definition}\label{d:6.2} Let $\pi$ be an irreducible representation in $\fk
  R^\one(\CG).$ It occurs as a subquotient in an unramified principal
  series $X(\chi)$ with $\CK$-spherical subquotient $\overline
  X(\chi).$ Via Theorem \ref{t:6.2}, to $\overline X(\chi)$ there
  corresponds a semisimple $\tau$-conjugacy class in $G.$ We will
  refer to 
  this class (or any representative of it) as the infinitesimal
  character of $\pi.$ 
\end{definition}
\noindent Notice that Definition \ref{d:6.2} is compatible, via Remark
\ref{r:6.1}, to our conventions from section \ref{sec:2}.

\begin{definition}We fix now an elliptic element $s_e\in T^\tau,$ and consider the
subcategory $\fk R^\one_{s_e}(\CG)\subset \fk R^\one(\CG)$ of representations of $\CG$ with infinitesimal character having elliptic
part $\tau$-conjugate to $s_e.$ 
\end{definition}

Let \begin{equation}G(s_e\tau)=\{g\in G: g
  s_e\tau(g^{-1})=s_e\}\end{equation} denote the twisted centralizer of
  $s_e$ in $G$. This is a potentially disconnected reductive group. Denote the identity component by $G(s_e\tau)_0$. Let
  $\CG(s_e\tau)$ denote a split $p$-adic group whose Langlands dual
  is $G(s_e\tau)$. In particular, this means that the group is the
split $\bF$-points of a {disconnected} group whose identity
component has root datum dual (in the sense of Langlands) to that of 
{$G(s_e\tau)_0$, and whose group of components is
isomorphic to $G(s_e\tau)/G(s_e\tau)_0.$}

Using the analysis after Lemma \ref{l:6.4}, based on a case-by-case
inspection of the tables in \cite{Ti}, we find that the graded algebra
at $s_e$ of the affine Iwahori-Hecke algebras for $\CG$ is naturally
$*$-preserving isomorphic with the graded algebra at $1$ of 
$\CG(s_e\tau)$. When $s_e=1,$ we denote this group by $\CG(\tau).$ We obtain the following consequence.

\begin{corollary}
The categories $\fk R_{s_e}^\one(\CG)$ and $\fk
R_1^\one(\CG(s_e\tau))$ are unitarily equivalent. In particular, there
is a unitary equivalence for representations with Iwahori fixed
vectors and real infinitesimal character between $\CG$ and $\CG(\tau).$ 
\end{corollary}

\begin{example}The explicit cases for quasisplit groups and real
  infinitesimal characters 
are in Table \ref{t:1}.
For example, when $\CG$ is the quasisplit adjoint unitary group in four
variables ($\CG=PSU(2,2)$), we have $G=SL(4,\bC)$. The corresponding automorphism
$\tau$ is 
\begin{equation}
\tau(g)=J\cdot g\cdot J^t,\text{ where }J=\left(\begin{matrix} 0
  &I\\-I &0\end{matrix}\right).
\end{equation}
The torus $T$ is the diagonal one. If we choose $s_e=1,$ then
$G(\tau)=Sp(4,\bC),$ so the correspondence of unitarity is with
$\CG(\tau)=SO(5,\bF),$ but if we choose
$s_e=\text{diagonal}(1,1,-1,-1)$, then $G(s_e\tau)=SO(4,\bC)$, and the correspondence of unitarity
is with $\CG(s_e\tau)=SO(4,\bF).$
\end{example}

\begin{table}\caption{Nonsplit quasisplit unramified $\CG$ and corresponding split $\CG(\tau)$.\label{t:1}}
\begin{tabular}{|c|c|c|c|}
\hline
Type of $\CG$ &Label in \cite{Ti} &Order of $\tau$ &Type of $\CG(\tau)$\\
\hline
$A_{2n-1}$ &$^2\!A_{2n-1}'$ &$2$ &$B_n$\\
\hline
$A_{2n}$ &$^2\!A_{2n}'$ &$2$ &$C_n$\\
\hline
$D_n$ &$^2\!D_n$ &$2$ &$C_{n-1}$\\
\hline
$D_4$ &$^3\!D_4$ &$3$ &$G_2$\\
\hline
$E_6$ &$^2\!E_6$ &$2$ &$F_4$\\
\hline
\end{tabular}
\end{table}

\subsection{Ramified principal series for split groups}\label{sec:6.2}
A similar type of
correspondence can be achieved for ramified principal series using the
results of \cite{Ro}. We assume that $\CG$ is split over $\bF$ and  the same restrictions of the
characteristic $p$ as in \cite[section 3]{Ro}. Let $\chi:A\to\bC^\times$ be a ramified
character and set $^0\!\chi:~^0\!A\to\bC^\times,$ $^0\!\chi=\chi|_{^0\!A}$. One considers
the inertial class $\fk r=[A,\chi].$ This only depends on
$^0\!\chi.$ The irreducible objects in $\fk R^\fk r(\CG)$ are the
irreducible subquotients of minimal principal series $X(\chi'),$ where $^0\!\chi'=~^0\!\chi.$

In \cite[section 3]{Ro}, a $\fk r$-type $(J,\rho)$ is constructed, with $\rho$
one dimensional, and the structure of the corresponding Hecke algebras is
computed. Let $G$ be the complex group dual to $\CG.$  A semisimple element $\widehat{^0\!\chi}\in G$ is 
attached to $^0\!\chi$ \cite[section 8]{Ro}. We explain this construction next. Let $\C W_\bF$ be the Weil group and recall
the short exact sequence
\begin{equation}
1\longrightarrow I_\bF\longrightarrow \C W_\bF\longrightarrow 
\bZ\longrightarrow 1,
\end{equation}
where $I_\bF$ is the inertia group (open in $\C W_\bF$), and the cyclic
group $\bZ$ generated by a Frobenius element $\text{Frob}:x\mapsto x^q$ in
$\Gal(\overline\bF_q/\bF_q).$
The Weil-Deligne group is $\C W_\bF'=\bC\rtimes \C W_\bF,$
where $w\in \C W_\bF$ acts on $\bC$ by multiplication by $||w||$, the
norm of $w.$ In particular,  $\text{Frob}$ acts by multiplication by $q.$ 
 
A Weil homomorphism $\phi: \C W_\bF'\to G$ is a continuous
homomorphism satisfying {certain properties (see for example \cite[\S8.1]{Bo2})}. In particular, $\phi|_{\C
  W_\bF}$ should consist of semisimple elements and $\phi|_{\bC}$ of
unipotent elements. We say that $\phi$ is unramified if $\phi(I_\bF)$
consists of central elements of $G.$  

Let $\C W_\bF^{\text{ab}}$ denote the abelian quotient of $\C W_\bF.$ The
Weil homomorphisms that parameterize minimal principal series descend
to $\bC\rtimes\C W_\bF^{\text{ab}}.$ Let $I_\bF^{\text{ab}}$ denote the image of
$I_\bF$ in $\C W_\bF^{\text{ab}}.$ Recall that the reciprocity homomorphism
$\tau_\bF$ is an isomorphism of $\C W_\bF^{\text{ab}}$ onto $\bF^\times,$ and
induces an isomorphism of $I_\bF^{\text{ab}}$ onto $\bO^\times.$

To the character $^0\!\chi:~^0\!A\to
\bC^\times,$  we attach a homomorphism \begin{equation}\widehat{^0\!\chi}:
I_\bF^{\text{ab}}\to T\subset G.\end{equation}
This is the
unique homomorphism which makes the following diagram commutative for
any $\lambda\in X_*(A)=X^*(T)$:
\begin{equation}
\xymatrixcolsep{2pc} 
\xymatrixrowsep{1pc}
\xymatrix
{
I_\bF^{\text{ab}}\ar@{.>}[rr]^{\widehat{^0\!\chi}}\ar[d]_{\tau_\bF}&&T\ar[d]^\lambda \\
\bO^\times\ar[dr]_\lambda && \bC^\times\\
&{^0\!A}\ar[ur]_{^0\!\chi}&
}
\end{equation}
Now we consider 
\begin{equation}\label{phi1}\phi:\bC\rtimes\C W_\bF^{\text{ab}}\to G
  \text{ such that
$\phi|_{I_\bF^{\text{ab}}}=\widehat{^0\!\chi}.$}
\end{equation}
 Such homomorphisms
parameterize L-packets which have subquotients of minimal principal
series $X(\chi),$ with $\chi|_{^0\!A}=^0\!\chi.$ Define
\begin{equation}
C_G(\widehat{^0\!\chi})=\text{~centralizer in $G$ of the image of $\widehat{^0\!\chi}.$}
\end{equation}
\begin{lemma}[{\cite[section 8]{Ro}}]
$C_G(\widehat{^0\!\chi})$ is the centralizer of a single semisimple
  element in $G$.
\end{lemma}
By abuse of notation, we denote this semisimple element by
$\widehat{^0\!\chi}$ too.
 Let $\Psi_{\widehat{^0\!\chi}}$ be the root datum for the
identity component of the centralizer $C_G(\widehat{^0\!\chi})$, and let
$\Gamma$ be the group of components of $C_G(\widehat{^0\!\chi})$.

\begin{theorem}[{\cite[Theorem 6.3 and section 9]{Ro}}] The $\fk r$-type $(J,\rho)$ is affine, and
  the Hecke algebra $\CH(G,\rho)$ is naturally isomorphic to the extended affine Hecke algebra $\CH'(\Psi_{
\widehat{^0\!\chi}})=\CH(\Psi_{\widehat{^0\!\chi}})\rtimes\Gamma$ with
equal parameters.
\end{theorem}

\begin{corollary}
The categories $\fk R^\fk r(\CG)$ and $\CH(\CG,\rho)$-mod (equivalently,
$\CH'(\Psi_{\widehat{^0\!\chi}})=\CH(\Psi_{\widehat{^0\!\chi}})\rtimes\Gamma$-mod) are
unitarily equivalent (in the sense of Definition \ref{d:unitequiv}).
\end{corollary}

By combining this with Theorem \ref{t:6.1}, we have an
important consequence. Let $\CG'(^0\!\chi)$ be a split $p$-adic
group dual to $C_G(^0\!\chi)$. 
The Iwahori Hecke algebra of this group is naturally identified with
$\CH'(\Psi_{\widehat{^0\!\chi}}).$ By combining the previous corollary with Theorem \ref{t:6.1}, we have an
important consequence.

\begin{corollary}
The categories $\fk R^\fk r(\CG)$ and $\fk R^\fk o(\CG'(^0\!\chi))$
are unitarily equivalent via the equivalences of categories:
$$\fk R^\fk r(\CG)\xrightarrow {\cong}
\CH'(\Psi_{\widehat{^0\!\chi}})\text{-mod}\xleftarrow{\cong} \fk R^\fk o(\CG'(^0\!\chi)).$$
\end{corollary}

\medskip

Assume now that $\CG$ is adjoint. In this setting,  \cite{Re} gives the Deligne-Langlands-Lusztig classification for $\fk R^\fk r(\CG).$  
Notice that any parameter $\phi$ as in (\ref{phi1}) has the image in fact in
$G(\widehat{^0\!\chi})$. Denote by $\phi'$ the homomorphism
obtained by restricting the range of $\phi$ to
$G(\widehat{^0\!\chi})$. Then $\phi'$ is, by definition, an unramified
parameter for $\CG(^0\!\chi)$. Therefore, we have a one-to-one
correspondence of L-packets
\begin{equation}\label{eq:psi}
\Psi: \phi\longrightarrow \phi',\quad
\phi|_{I_\bF^{\text{ab}}}=\widehat{^0\!\chi},\ \phi'\text{ unramified for }\CG(^0\!\chi).
\end{equation}
The assumption that $\CG$ is adjoint implies that $\CG(^0\!\chi)$ is
connected. Following \cite{Re}, the correspondence (\ref{eq:psi}) encodes the bijection between 
subquotients for the $^0\!\chi$-ramified principal series of
$G(\bF)$ and subquotients $\phi'$ of the unramified principal series
of $G(^0\!\chi)$. More precisely, let $A_{G}(\phi)$ and
$A_{G(\widehat{^0\!\chi})}(\phi')$ denote the component groups of
the centralizers of the images of $\phi$ and $\phi'$ respectively. Let
$\C B_{G(\widehat{^0\!\chi})}^\phi$ and $\C
B_{G(\widehat{^0\!\chi})}^{\phi'}$ denote the varieties of Borel
subgroups fixed by the images of $\phi$ and $\phi'$ respectively. We
say that a  representation of the component group is of Springer type
if it appears in the action on the Borel-Moore homology of
these varieties. Then there is a natural isomorphism 
\begin{equation}A_{G}(\phi)\cong
A_{G(\widehat{^0\!\chi})}(\phi'),
\end{equation} which induces a bijection
$\Psi$ 
of the component group representations of Springer type.

Then the reformulation of the corollaries in section \ref{sec:6.2}, in
the particular case when $\CG$ is adjoint, is that
the correspondence $\Psi$ of (\ref{eq:psi}) restricted to elements of Springer type, gives a one-to-one correspondence between hermitian and unitary representations, respectively.

\ifx\undefined\bysame
\newcommand{\bysame}{\leavevmode\hbox to3em{\hrulefill}\,}
\fi

\end{document}